%% file: quad_asymptotic_arxiv.tex
\documentclass[11pt]{article}

\pdfoutput=1
\usepackage[T2A,T1]{fontenc}
\usepackage[utf8]{inputenc}
\usepackage[russian,english]{babel}

\usepackage[labelfont=bf]{caption}
\usepackage[top=50pt, bottom=50pt, left=50pt, right=50pt]{geometry}

\usepackage{graphicx,amsmath,amssymb,amsthm}
\usepackage{epstopdf}
\usepackage{pifont, url, romanbar,subfig,listings}
\usepackage{tikz}
\usepackage{pgfplots}
\usepackage{verbatim}

\newlength{\figurewidth}
\newlength{\figureheight}

\usepackage{hyperref}
\usepackage[poorman]{cleveref}
\usepackage{autonum} 

\usepackage{algorithm}
\usepackage{algpseudocode}

\allowdisplaybreaks

\newcommand{\R}{{\operatorname{right}}}
\renewcommand{\L}{{\operatorname{left}}}
\renewcommand{\O}{{\operatorname{outer}}}

\newcommand{\ql}{L}
\newcommand{\qj}{J}
\newcommand{\qh}{H}

\newcommand{\rl}{{\scriptsize \Romanbar{4}} }

\newcommand{\dint}{{\rm d}}

\newtheorem{lemma}{Lemma}[section]

\usetikzlibrary{patterns,arrows,decorations.markings}

\numberwithin{equation}{section}
\usepackage{remreset} \makeatletter \@removefromreset{lemma}{section} \makeatother

\newcommand{\Ri}{\Romanbar{1}}
\newcommand{\Ro}{\Romanbar{2}}
\newcommand{\Rright}{\Romanbar{3}}
\newcommand{\Rleft}{\Romanbar{4}}

\newcommand{\ee}{\mathrm{e}}

\DeclareMathOperator{\airy}{Ai}

\addto\captionsenglish{
    \crefname{figure}{Figure}{Figures}
    \Crefname{figure}{Figure}{Figures}
    \crefname{table}{Table}{Tables}
    \crefname{section}{\S}{\S}
    \Crefname{section}{\S}{\S}
    \crefname{equation}{}{}
    \Crefname{equation}{}{}
    \crefname{remark}{Remark}{Remarks}
    \crefname{algorithm}{Algorithm}{Algorithms}
    \crefname{appendix}{Appendix}{Appendices}
}

\title{Arbitrary-order asymptotic expansions of {G}aussian quadrature rules with classical and generalised weight functions}
\author{
Peter Opsomer\footnote{peter.opsomer@gmail.com (corresponding author)}\\ 
Department of Computer Science\\
KU Leuven, Belgium
\and
Daan Huybrechs\footnote{daan.huybrechs@kuleuven.be}\\
Department of Computer Science\\
KU Leuven, Belgium
}

\begin{document}

\maketitle

\begin{abstract}
Gaussian quadrature rules are a classical tool for the numerical approximation of integrals with smooth integrands and positive weight functions. We derive and expicitly list asymptotic expressions for the points and weights of Gaussian quadrature rules for three general classes of positive weight functions: analytic functions on a bounded interval with algebraic singularities at the endpoints, analytic weight functions on the halfline with exponential decay at infinity and an algebraic singularity at the finite endpoint, and analytic functions on the real line with exponential decay in both directions at infinity. The results include the Gaussian rules of classical orthogonal polynomials (Legendre, Jacobi, Laguerre and Hermite) as special cases. We present experiments indicating the range of the number of points at which these expressions achieve high precision. We provide an algorithm that can compute arbitrarily many terms in these expansions for the classical cases, and many though not all terms for the generalized cases.
\end{abstract}

\section{Introduction} \label{SquadrPref}

The topic of this paper is a methodology to obtain high-order asymptotic expansions of nodes and weights of Gaussian quadrature rules. As a significant example, we fully derive results for generalizations of the classical Jacobi and Laguerre weights. The resulting explicit expressions for the classical cases are listed in Appendix~\ref{app:explicit}, with more terms then appeared before in literature. The proposed methodology is based on the asymptotic analysis of the corresponding orthogonal polynomials for large degree. It results in explicit formulae for the nodes and weights which can simply be evaluated one by one, leading to an efficient ${\mathcal O}(n)$ computational scheme for the construction of Gaussian quadrature. In this paper we focus mainly on the essential principles. Detailed computations, which are quite technical in nature, can be found in the phd thesis of the first author~\cite{opsomer2018phd}.

Different expansions are presented for three distinct regions of the integration domain in which the asymptotic behaviour is different: near the endpoints (called a \emph{hard edge} in the literature on asymptotic analysis), near infinity (\emph{soft edge}) and in the interior of the domain (the \emph{bulk}). An algorithm is presented to obtain an arbitrary number of terms. It has been implemented in symbolic software, in which parameters may be present, and in numerical software, for specific numeric choices of those parameters and without a further need for symbolic computations. Our implementation is accompanied by heuristical approaches to choose the appropriate asymptotic expansion and number of terms.

Historically, the main algorithm for the computation of Gaussian quadrature rules has been based on the computation of the eigenvalues of a tridiagonal matrix (the Jacobi matrix), which involves the recurrence coefficients of the three-term recurrence relation of the orthogonal polynomials~\cite{GautschiOPQ}. The computation can be performed in ${\mathcal O}(n^2)$ operations using the Golub--Welsch algorithm \cite{GolubWelsch}. Since then, algorithms with lower computational complexity for large $n$ have started to appear for a number of cases, starting with the Glaser--Liu--Rokhlin algorithm \cite{glaser} and including at least \cite{BogaertIterationFree,bogaert2012legendre,bremer,GilSeguraQuadr,GST,HaleTownsend,OlverTrogdon_RH3}. A brief history of fast algorithms for the special case of Gauss--Legendre quadrature rules is described in \cite{SiamNewsRace}.

\subsection{Asymptotic expansions of orthogonal polynomials and numerical aspects}
Asymptotic analysis of orthogonal polynomials has witnessed significant progress in the last two decades. Large-degree asymptotic expansions can be obtained from integral representations, differential equations satisfied by the polynomials (see \cite{gil2021polynomial} and references therein), or the associated Riemann--Hilbert problem (see in particular~\cite{KMcLVAV,Vanlessen} for Jacobi and Laguerre polynomials).

From a computational point of view, asymptotic expansions of orthogonal polynomials present two important advantages: they become increasingly accurate as the degree $n$ increases, and their evaluation time is independent of $n$.\footnote{Strictly speaking, for this statement we assume that all functions involved can be evaluated in a time independent of $n$, which in some cases may require lookup tables \cite{johIEEE}.} In this sense, they compare favourably to other methods for computing orthogonal polynomials. On the other hand, the terms in the expansion become increasingly involved with increasing order. Moreover, asymptotic expansions in general do not converge for a fixed value of $n$ by adding terms. A typical recommendation is to truncate the expansion after the smallest term, after which the accuracy often deteriorates \cite[p. 182-187]{DavisRabinowitz}. Still, arbitrary accuracy can not be achieved and the availability of error bounds is, unfortunately, lagging behind the construction of the expansions themselves.

Gaussian quadrature rules can be computed from these expansions most simply by numerical rootfinding \cite{bogaert2012legendre,HaleTownsend}. The classical quadrature rules with $n$ points may be computed in ${\mathcal O}(n)$ cost. However, asymptotic expansions can also be found directly for the nodes and weights themselves. For the classical Gauss rules, this was achieved in \cite{BogaertIterationFree,GilSeguraQuadr,GST} and the thesis \cite{opsomer2018phd}. These expansions have ${\mathcal O}(n)$ cost as well, with a constant that is often (but not always) smaller than that of the rootfinding class of methods. Care has to be taken in the efficient evaluation of certain special functions and in various kinds of precomputations. Therefore, whether or not direct asymptotic expansions are faster is not a priori guaranteed and still depends on several implementation aspects, but they offer an appealing simplicity.

Other methods with linear complexity have been described, based on the numerical solution of Riemann--Hilbert problems \cite{OlverTrogdon_RH3,bookRH} and on efficient rootfinding methods for non-oscillatory differential equations \cite{bremer,bremerYang}. Particularly interesting is that these often can achieve arbitrary accuracy, see for example \cite{gil2019fastreliable} for an efficient method to compute Gauss--Laguerre and Gauss--Hermite rules to very high accuracy.

\subsection{Scope of the paper and main results} 
We develop a general approach for direct asymptotic expansions of Gaussian quadrature rules based on the Riemann--Hilbert (RH) problem of the associated orthogonal polynomials. Riemann--Hilbert problems provide a uniform and general framework in which to analyze orthogonal polynomials, not restricted to the classical cases. We have previously described the numerical computation of arbitrary order asymptotic expansions for orthogonal polynomials for Jacobi--type weights \cite{jacobi} and Laguerre-type weights \cite{laguerre}. This was based on the asymptotic analysis of the associated RH problems by Kuijlaars, McLaughlin, Van Assche and Vanlessen in \cite{KMcLVAV} for the Jacobi case and by Vanlessen in \cite{Vanlessen} for the Laguerre case. It resulted in high-order expansions in four different regions of the complex plane in which the polynomials exhibit different asymptotic behaviour. These expansions form the starting points for the expansions of the current paper of the associated Gaussian quadrature rules.

The orthonormal polynomials associated with Gaussian quadrature satisfy the orthogonality conditions
\begin{equation}
	\int_a^b p_n(x) p_k(x) w(x)dx = \begin{cases} 1, & k=n, \\ 0, & k \neq n. \end{cases}
\end{equation} 
The nodes $x_k$ are the zeros of $p_n$. The weights, sometimes referred to as Christoffel numbers, can be expressed in various ways (see \cite[p. 35]{DavisRabinowitz}, \cite[p. 323]{Hildebrand}, \cite[\S 1]{GolubWelsch}). One useful expression is
\begin{equation} 
	w_k = \frac{-\gamma_{n+1}}{\gamma_n p_n^\prime(x_k) p_{n+1}(x_k) } = \frac{\gamma_n}{\gamma_{n-1} p'_n(x_k) p_{n-1}(x_k)} > 0. \label{Eweights} 
\end{equation}
Here, $\gamma_n$ is the leading order coefficient of the orthonormal polynomial $p_n$.

We consider integrals with weight functions for three canonical types of quadrature rules:
\begin{itemize}
\item The first are \textbf{modified Jacobi--type weight functions} of the form
\begin{equation}
\label{Ejacobi_weight}
 w(x) = (1-x)^\alpha (1+x)^\beta h(x), \qquad x \in [-1,1]
\end{equation}
where $h$ is a strictly positive analytic function that represents the modification and $[a,b] = [-1,1]$.

\item The second are \textbf{modified Laguerre-type weight functions} of the form
\begin{equation}
\label{Elaguerre_weight}
 w(x) = x^{\alpha} \ee^{-Q(x)}, \qquad x \in [0,\infty),
\end{equation}
where $[a,b)=[0,\infty)$ and the standard Laguerre case corresponds to $Q(x)=x$. Here, unlike in the Jacobi case, more general functions $Q(x)$ are restricted as detailed further on. Our results are most explicit for monomials $Q(x)=x^m$.

\item Finally, we consider \textbf{modified Hermite--type  weight functions}
\begin{equation}
\label{Ehermite_weight}
 w(x) = \ee^{-Q(x^2)}, \qquad  x \in (-\infty, \infty),
\end{equation}
where $(a,b) = (-\infty,\infty)$ and where $Q(x)=x$ corresponds to the classical Gauss--Hermite case. The functions $Q(x)$ we consider are those of the Laguerre case.
\end{itemize}
The Hermite--type rules in the form above can be obtained from the Laguerre--type rules by a change of variables, hence we focus mostly on the first two cases. The substitution entails the specific values $\alpha=\pm 1/2$ which results in a simplification of the Laguerre expansions.

Numerical experiments indicate that the errors of the expansions decrease at the expected rate for increasing $n$, although no error bounds are provided in this paper. The expansions in the current paper hold in principle for any fixed value of $\alpha$ and $\beta$. However, when those parameters are large, accuracy of the expansions may decrease and any beneficial effect may require very high $n$. This is a known problem, and in that setting we recommend exploring different expansions, for example based on asymptotics with a varying weight (varying with $n$) as studied in \cite{varying,KMF_Jacobi,largeNeg,asyZerNegLag,varCompl}.

We have implemented an algorithm to compute asymptotic expansions for nodes and weights with arbitrary many terms in the symbolic software package {\sc Sage}. By making all steps of the algorithm explicit, i.e. not requiring symbolic treatment, the algorithm can also be implemented in a more efficient non-symbolic language. The resulting expressions themselves have to be computed only once, and can be implemented in any language. The implementation that we make available includes versions in {\sc Fortran}, {\sc Matlab} and {\sc Julia}.\footnote{The code \cite{codeQuadr} is available on our departmental homepage \url{https://nines.cs.kuleuven.be/software/} and in the GitHub repository \url{https://github.com/daanhb/PolynomialAsymptotics.jl}. Several results of this paper have also been incorporated in the Julia package \texttt{FastGaussQuadrature.jl} maintained by Alex Townsend and in Chebfun~\cite{chebfun}.} Some limitations are present in the case of modified Laguerre, as the expansions involve several integrals that have to be evaluated numerically for specific numerical choices of the parameters. This renders a generic listing of expansions rather difficult.

\subsection{Overview of the paper}

In order to understand the structure of the expansions, as well as their derivation, we recall the Riemann--Hilbert background of the orthogonal polynomials in \cref{Sprelim}.
We outline the procedure to turn expansions of polynomials into expansions of the nodes and weights in \cref{Sstrat}. We describe the computation of special functions appearing in our expansions and necessary precomputations in \cref{SquadrComputSpec}. The explicit expressions themselves are listed in \cref{SLaguerre}, \cref{SJacobi} and \cref{SquadrGH}, with the classical cases listed in Appendix~\ref{app:explicit}. 
We illustrate the statements with numerical experiments throughout the paper.

\section{Preliminaries} \label{Sprelim}

We briefly recall the background of the asymptotic expansions for Jacobi--type and Laguerre--type orthogonal polynomials as derived in \cite{jacobi,laguerre}, see also \cite{opsomer2018phd}. In particular, we define the quantities that arise from the modified weight functions \eqref{Ejacobi_weight}--\eqref{Elaguerre_weight} and that appear in the final expansions: the constants $c_k$ and $d_k$, the phase functions $\lambda_n(z)$ and $\zeta_n(z)$ and, in case of Laguerre--type weights, the constant $\beta_n$. These quantities depend on the analytic modifier $h(x)$ in \eqref{Ejacobi_weight} or on the function $Q(x)$ in \eqref{Elaguerre_weight} and \eqref{Ehermite_weight}.

\subsection{Riemann-Hilbert analysis} \label{SRHA}

The explicit expansions in \cite{jacobi,laguerre} are derived from the Deift-Zhou nonlinear steepest descent analysis \cite{DeiftZhou} of a Riemann--Hilbert problem associated with orthogonal polynomials \cite{fokas1992isomonodromy,deift2000rh}. The asymptotic analysis for Jacobi-type weight functions of the form \cref{Ejacobi_weight} was carried out in \cite{KMcLVAV}, while the description for Laguerre-type weight functions \cref{Elaguerre_weight} is in \cite{Vanlessen}.

The Riemann-Hilbert problem has a $2\times 2$ complex matrix-valued solution. The main result in \cite{jacobi,laguerre} has been the explicit asymptotic approximations to a matrix-valued function $R(z)$, defined in the complex plane. For large $n$, this matrix is the identity matrix plus an asymptotic series,
\begin{equation}
 R(z) \sim I + \sum_{k=1}^{\infty} \frac{R_k(z)}{n^{k}}, \qquad n \rightarrow \infty. \label{asympRn}
\end{equation}
The expansions of the polynomials are formulated in terms of $R(z)$ and, hence, additional terms $R_k(z)$ give rise to additional terms in the expansions of the polynomials. This happens via an explicit and invertible sequence of transformations.

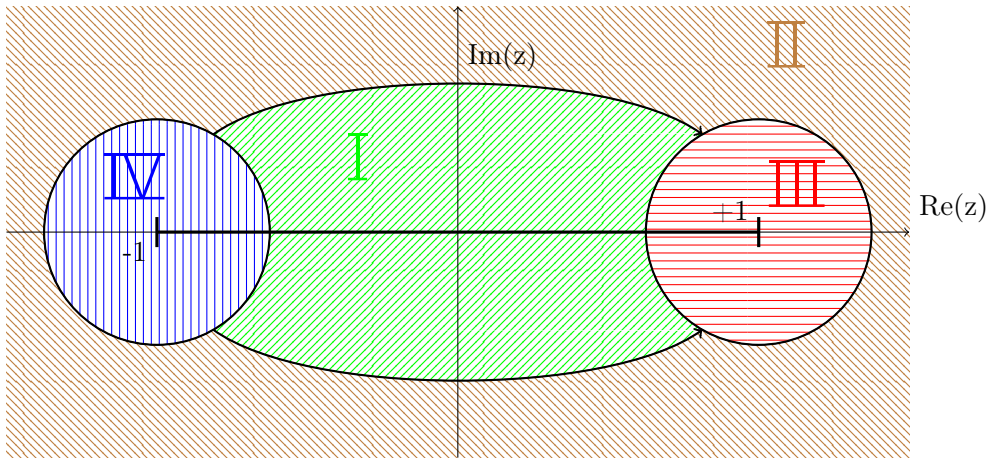
\begin{figure}[t]
\begin{center}
\begin{tikzpicture}
        \fill[pattern=north west lines, pattern color =brown] (0,4) rectangle (12, -2);
        \fill[white] (2.75,2.32) rectangle (9.25, -0.32);
        \draw[->] (2,1) [fill=white] circle [radius=1.5];
        \draw[->] (10,1) [fill=white] circle [radius=1.5];
        \draw[->,fill=white] (2.75, 2.3) .. controls(4,3.2) and (8, 3.2) .. (9.25, 2.3);
        \draw[->,fill=white] (2.75, -0.3) .. controls(4,-1.2) and (8, -1.2) .. (9.25, -0.3);

        \fill[pattern=north east lines, pattern color =green] (2.75,2.3) rectangle (9.25, -0.3);
        \draw[->] (2,1) [fill=white] circle [radius=1.5];
        \draw[->] (10,1) [fill=white] circle [radius=1.5];
        \draw[->] (10,1) [thick,pattern=horizontal lines, pattern color =red] circle [radius=1.5];
        \draw[->] (2,1) [thick,pattern=vertical lines, pattern color =blue] circle [radius=1.5];
        \draw[->,thick,pattern=north east lines, pattern color =green] (2.75, 2.3) .. controls(4,3.2) and (8, 3.2) .. (9.25, 2.3);
        \draw[->,thick,pattern=north east lines, pattern color =green] (2.75, -0.3) .. controls(4,-1.2) and (8, -1.2) .. (9.25, -0.3);

        \draw [very thick] (2,1)--(10,1);
        \draw[->] (6,-2)--(6,4);
        \draw (6,3) node[above right] {Im(z)};
        \draw[->] (0,1)--(12,1);
        \draw (12,1) node[above right] {Re(z)};
        \draw (2,1) node[below left] {-1};
        \draw (10,1) node[above left] {+1};
        \draw[very thick] (2,0.8)--(2,1.2);
        \draw[very thick] (10,0.8)--(10,1.2);
        \draw[green] (4.8,2) node {\Huge{\Romanbar{1} } };
        \draw[brown] (10.5,3.5) node {\Huge{\Romanbar{2} } };
        \draw (10,1.2) node[red,above right] {\Huge{\Romanbar{3} } };
        \draw (2.5,1.3) node[above left,blue] {\Huge{\Romanbar{4} } };
\end{tikzpicture}
\end{center}
\caption{Regions of the complex plane in which the Jacobi-type polynomials have different asymptotic expansions: the lens (\Ri), the outer region (\Ro) and the right and left disks (\Rright~and \Rleft).}
\label{Fregions}
\end{figure}

The asymptotic analysis is different in different regions of the complex plane. The regions are illustrated for the case of Jacobi-type polynomials in \cref{Fregions}. There are four types of asymptotic expansions of the orthogonal polynomials $p_n(z)$: (\Ri) inner asymptotics for $z$ in the so-called `lens' around the interval, (\Ro) outer asymptotics valid for $z$ away from the lens and the disks around the endpoints, (\Rright) boundary asymptotics valid for $z$ near the rightmost endpoint $z=+1$ and (\Rleft) boundary asymptotics near the left endpoint $z=-1$. The matrix-valued function $R(z)$ has different expressions in these regions, which we label $R^{\O}(z)$ away from the disks, and $R^{\R}(z)$ and $R^{\L}(z)$ in the right and left disks respectively. Only the expansions in regions $(\Ri)$, $(\Rright)$ and $(\Rleft)$ contain zeros of the polynomials and they give rise to expansions for the bulk of the zeros in $(-1,1)$, and for the extreme zeros near $\pm 1$.

The figure with four regions is qualitatively similar for Laguerre-type and Hermite-type polynomials, albeit with different endpoint regions. We elaborate on the differences further on.

Mathematically, the regions shown in \cref{Fregions} are of arbitrary size. Depending on how they are chosen, any given point $z\in\mathbb{C}$ can in principle belong to several regions. As a result, the expansions for the different regions have overlapping regions in which they are valid. In terms of implementation, the choice of the expansion is relevant as the accuracy can differ. We formulate heuristic choices of the expansions in later sections of this paper. We experimentally observed that, when two expansions apply at the same point, the differences in the corresponding relative errors for large $n$ are fairly modest, in most cases not exceeding a factor of $2$. The exception is when one expansion approaches the boundaries of its domain of validity: for example, the expansion for the outer region is not accurate near the interval. Within a radius of about $0.2$ around an endpoint, the asymptotic expansion for that disk is orders of magnitude more accurate than the three other ones.

\subsection{Jacobi-type polynomials}

The asymptotic expansion for the orthonormal Jacobi-type polynomials valid in region (\Ri) in the complex plane, as formulated in \cite{jacobi}, is\footnote{We note that the formulation in \cite{jacobi} is for the monic orthogonal polynomials, which differs from the formulation here in the leading order coefficient $\gamma_n$. The latter can be expanded asymptotically too, see \cite[(26)]{jacobi}.}
\begin{equation}
	p_n(z) = \frac{2^{\frac12 - n} \, \gamma_n}{\sqrt{w(z)} \, (1-z^2)^{1/4} } \begin{pmatrix} 1 \\ 0 \end{pmatrix}^T R^{\O}(z) \begin{pmatrix} D_\infty & 0 \\ 0 & -\frac{i}{D_\infty} \end{pmatrix} 
	\begin{pmatrix} \cos \left( \lambda_n(z) + \frac12  \arccos z \right) \\ \cos\left( \lambda_{n}(z) - \frac 12 \arccos z \right) \end{pmatrix}. \label{EpiInt}
\end{equation}
Here, $\gamma_n$ is the leading order coefficient (with $z^n$) of the polynomial and $D_\infty$ is a constant,
\[
 D_{\infty} =  2^{-\frac{\alpha+\beta}{2}}\exp\left( \frac{1}{4\pi i}\oint \frac{\log h(\zeta)}{(\zeta^2-1)^{1/2}} d\zeta \right),
\]
where the contour encircles the interval $[-1,1]$ and $h(z)$ is the additional function that appears in the Jacobi-type weight function \cref{Ejacobi_weight}. The constant $D_\infty$ is the limit for $z \to \infty$ of the Sz{\"e}go function $D(z)$ of the orthogonal polynomial. Note that for the classical Jacobi case, where $h(x) = 1$, we have $D_\infty=2^{-\frac{\alpha+\beta}{2}}$. The leading order term of \cref{EpiInt} corresponds to substituting $R^{\O}(z)$ by $I$, the $2 \times 2$ identity matrix.

The trigonometric behaviour of the polynomials in this region is determined by the phase function
\begin{equation}
 \lambda_n(z)  = \left(n + \frac{\alpha + \beta}{2}\right)\arccos z   -\frac{\pi}{4} -\frac{\alpha\pi}{2} + \frac{(1-z^2)^{\frac12}}{4 \pi i}  \oint \frac{\log h(\zeta)}{(\zeta^2-1)^{1/2}}\frac{d\zeta}{\zeta-z}. \label{Elambda}
\end{equation}
The phase depends on the $h(x)$ function via a contour integral as well, which in this case should encircle the point $z$. This contour integral is a function of $z$ and can be expanded around $z= \pm 1$, for example using the power series $\sum_{n=0}^{\infty} c_n (z-1)^n$ and $\sum_{n=0}^{\infty} d_n (z+1)^n$. These series have coefficients given by
\begin{align}
	c_k & = \frac{1}{2\pi i} \oint_{\gamma} \frac{\log h(\zeta)}{(\zeta^2-1)^{1/2}}\frac{\dint \zeta}{(\zeta-1)^{k+1}}, \label{Ecn} \\
	d_k & = \frac{1}{2\pi i} \oint_{\gamma} \frac{\log h(\zeta)}{(\zeta^2-1)^{1/2}}\frac{\dint \zeta}{(\zeta+1)^{k+1}}, \label{Edn}
\end{align}
for $k \geq 0$. The coefficients enter the asymptotic expansions of the polynomials, and of the points and weights in this paper. We refer to \cite{jacobi} for comments on their computation.

Around the endpoints the polynomials exhibit behaviour that relates to the Bessel function. Each endpoint $\pm 1$ is independent of $n$ and for this reason is sometimes referred to as a `hard edge'. In a disk around the right endpoint $z=1$ the expression given in \cite{jacobi} is
\begin{align}
	p_n(z) &= \gamma_n \frac{ 2^{-n} \sqrt{n \pi \arccos z} }{\sqrt{w(z)} \, (1-z^2)^{1/4}  } \begin{pmatrix} 1 \\ 0 \end{pmatrix}^TR^{\R}(z) \begin{pmatrix} D_\infty & 0 \\ 0 & -\frac{i}{D_\infty} \end{pmatrix}  \times \label{Epiboun}  \\
	& \begin{pmatrix} \cos\left(\zeta(z) + \frac12 \arccos z\right) J_\alpha(n \arccos z)  + \sin\left(\zeta(z) + \frac12 \arccos z \right) J_\alpha'(n \arccos z) \\
	     \cos\left(\zeta(z) - \frac12 \arccos z \right) J_\alpha(n \arccos z) + \sin\left(\zeta(z) - \frac12 \arccos z \right) J_\alpha'(n \arccos z) \end{pmatrix}. \nonumber
\end{align}
This expression may seem to be undefined near $z=1$; however, the singularities cancel analytically with other terms. In our implementation for the Jacobi case \cite{ninesJac}, we have also included series expansions around the endpoints for that reason. Here, the phase function $\zeta(z)$ is given by
\begin{equation}
 \zeta(z)  =\frac{\alpha+\beta}{2}\arccos z +\frac{(1-z^2)^{1/2}}{4\pi i}\oint_{\gamma} \frac{\log h(\zeta)}{(\zeta^2-1)^{1/2}}\frac{d\zeta}{\zeta-z},  \label{Ezeta}
\end{equation}
where the contour integral is similar as before. The expansion for the left disk is entirely similar due to the symmetry relation $P_n^{(\alpha,\beta)}(-x) = (-1)^n P_n^{(\beta,\alpha)}(x)$ \cite[Tab. 18.6.1]{DLMF}. The endpoint expansions remain structurally similar even when $h(x)$ is not symmetric, but in that more general case the symmetry involves substituting $h(x)$ for $h(-x)$.

Both expansions have the form of a prefactor consisting of a normalization factor and an envelope function, multiplied by an oscillatory part consisting of $R(z)$ and oscillatory factors. Note that the function $h$ appears in the phase of the oscillations, via a contour integral; in $R(z)$, via $c_k$ and $d_k$; and in the prefactor, via the weight function $w(z)$. The prefactor does not affect the points $x_k$, but it does affect the weights.

\subsection{Laguerre-type polynomials}

The structure of the expansions for Laguerre-type polynomials is very similar to that of Jacobi-type polynomials. We do not repeat the expressions here, but refer to \cite{laguerre,opsomer2018phd}. We do recall the context and the expressions that are necessary for understanding and implementing the expansions of this paper.

Recall the generalised Laguerre weight function \cref{Elaguerre_weight}. We consider three cases for the function $Q(x)$:
\begin{itemize}
 \item a monomial: $Q(x) = q_m x^m+q_0$,
 \item a general polynomial: $Q(x) = \sum_{k=0}^m q_k x^k$,
 \item a general function $Q(x)$ analytic on the integration interval.
\end{itemize}
The function $Q(x)$ should be such that all moments $\int_0^\infty x^\alpha x^k e^{-Q(x)} {\rm d}x$ are finite, which implies at least that
\[
\lim_{x \to \infty} Q(x)  = +\infty.
\]
For polynomial $Q$, it is sufficient that the leading order coefficient $q_m$ is positive. Here, as we shall see the expansions depend on certain contour integrals involving $Q$ which have to be evaluated by other means, and we present only limited results.

The left endpoint $x=0$ of the Laguerre-type polynomials is very similar to the left endpoint of Jacobi-type polynomials: it is a hard edge that gives rise to expansions in terms of the Bessel function and its derivative.

The behaviour is quite different for $x \to \infty$, a so-called \emph{soft edge}. The zeros of the Laguerre-type polynomials accumulate on a finite interval $(0,1)$ after an $n$-dependent rescaling of the variable
\[
 x = \beta_n z,
\]
where $\beta_n$ is the \emph{Mhaskar-Rakhmanov-Saff (MRS) number} (see \cite[\S3]{laguerre}). The local asymptotic behaviour near the soft edge is given in terms of the Airy function and its derivative. This leads to higher order poles in the derivations and thus also to longer formulae for the $R_k(z)$.

For classical Laguerre polynomials, the MRS number is $\beta_n = 4n$. Indeed, the largest zero of the Laguerre polynomial of degree $n$ scales approximately as $4n$. For general monomial functions $Q(x)=q_m x^m$, we have
\[
 \beta_n = n^{1/m} \left( \frac{m \, q_m \, A_m}{2}\right)^{-\frac{1}{m}},
\]
with
\begin{equation} \label{EAm}
A_k = 4^{-k}
\begin{pmatrix} 2k \\ k \end{pmatrix}.
\end{equation}
In the more general polynomial case, $\beta_n$ itself has an expansion in powers of $n^{-\frac{1}{m}}$. For general $Q$, the MRS number is such that
\begin{equation}
 \label{Emrs}
 2\pi n = \int_0^{\beta_n} Q'(x) \sqrt{ \frac{x}{\beta_n - x} } \, {\rm d} x.
\end{equation}
A full expansion of $\beta_n$ can not be derived at this level of generality. However, equation \cref{Emrs} can be solved numerically, such that for each $n$ a value of $\beta_n$ can be computed. These values appear in the expansions. We refer to \cite[\S3]{laguerre} for more considerations on the computation and expansion of the MRS numbers.

In the asymptotic expansions of Laguerre-type orthogonal polynomials, the generalisation of the weight function $Q(x)$ appears in the argument of the Airy and Bessel functions, but not in the trigonometric functions that are multiplied by them, while it was the other way around in \cref{Epiboun} for the Jacobi case. This is due to choices made in the derivation of the Riemann-Hilbert analysis.

Finally, the modification $Q(x)$ results in $c_k$ and $d_k$ coefficients that are analogous to \cref{Ecn}. For the Laguerre-type case, they are
\begin{align}
	c_k & = \frac{\beta_n}{2\pi i n} \oint \frac{\sqrt{y} Q'(\beta_n y)}{\sqrt{y-1}y^{k+1}} \, {\rm d}y, \label{Ecn2} \\
	d_k & = \frac{\beta_n}{2\pi i n} \oint \frac{\sqrt{y} Q'(\beta_n y)}{\sqrt{y-1}(1-y)^{k+1}} \, {\rm d}y, \label{Edn2}
\end{align}
where the contours should enclose the interval $[0,1]$. These coefficients, too, appear in the asymptotic expansions for the nodes and weights.

\section{Methodology for the asymptotic expansions of $x_k$ and $w_k$} \label{Sstrat}

The process of inverting an expansion in order to produce expansions for its zeros is, in principle, well known. However, in practice, the computations involved are lengthy and laborious. Here, rather than detailing all computations, we illustrate the general principles. This allows to describe the structure of the expansions themselves, as well as the scope and limitations of our implementation.

We proceed formally. Consider an asymptotic expansions of the form
\begin{equation}
\label{Egeneral_expansion}
 u(n,x) \sim \sum_{k=0}^\infty u_k(x; n) \, n^{-\mu_k},
\end{equation}
where $\{\mu_k\}_{k=0}^\infty$ a strictly increasing sequence. Though the functions $u_k$ may depend on $n$, it is assumed they do not grow with $n$, such that the truncation of the right hand side in \cref{Egeneral_expansion} after $T$ terms has an error that decays like ${\mathcal O}(n^{-\mu_{T+1}})$.

\subsection{Leading order term}

To leading order, the roots of $u(n,x)$ are determined by the roots of the first term $u_0(x; n)$ in \cref{Egeneral_expansion}. This implies that we need to find all solutions to the equation
\begin{equation}
\label{Eleading_order}
 u_0(x; n) = 0.
\end{equation} 
We enumerate the roots as $t_k$ with $k=1,2,\ldots$, such that $t_k < t_{k+1}$.

The nature of this equation, and the difficulty of finding its solution, depends on the asymptotic expansion and on the special functions it contains. The leading order terms of the expansions in this paper include trigonometric functions (for the bulk regions), Bessel functions (for a hard edge) or Airy functions (for a soft edge). That is the reason why the roots of the Bessel function, for example, appear in the expansions for the nodes $x_k$ near the endpoints in the Jacobi-type case.

The more general modified weight functions also lead to more complicated leading order terms. This is a factor that significantly contributes to the difficulty of providing explicit asymptotic expansions for these more general cases. Explicit expressions appearing further on in this paper that specialize \cref{Eleading_order} include \cref{EtranscStdLag} for the bulk region in the case of Laguerre polynomials, and \cref{EtkBulkJacClassical} and \cref{EtkBulkJacModif} for the bulk region for Jacobi polynomials with and without the additional function $h(x)$ in the weight.

\subsection{Subsequent terms}

We formally consider a full expansion of the zeros $x_k$ of the form
\begin{equation}
\label{Eformal_root}
 x_k \sim \sum_{l=0}^\infty a_{k,l} \, n^{-\nu_l},
\end{equation}
where $\{\nu_l\}_{l=0}^\infty$ is a (suitably chosen) strictly increasing sequence.  One subsitutes \cref{Eformal_root} into \cref{Egeneral_expansion} and then proceeds by expanding all functions $u_k$ around $x_k$. This results again in an asymptotic expansion in inverse powers of $n$. Next, the terms of this expansion are equated to zero one by one.

Since the functions $u_k$ are known, the unknowns are the $a_{k,l}$ coefficients. These can be computed one by one, starting with $l=0$. Once the value of $a_{k,l}$ has been found, it can be substituted back into the expansion. Equating the expression multiplying a higher power of $n$ to zero yields a linear equation for the next coefficient $a_{k,l+1}$. In contrast to the first term to find $t_k$, for which \cref{Eleading_order} has to be solved, the computation of subsequent terms is a linear problem.

The above represents the general methodology. 
For Jacobi, in case of the lens, we use the notation
\[
 x_k = t_k + \sum_{l=1}^\infty z_{1,l+1} n^{-l},
\]
and we solve recursively for the values of $z_{1,l+1}$. The expression is similar for Laguerre in the lens,
\[
 \frac{x_k}{\beta_n} = t_k + \sum_{l=1}^\infty z_{1,l+1} n^{-l},
\]
where $\beta_n$ is the MRS number. In case of the left boundary, we use
\[
 x_k = -1 + \frac{t_k}{n^2} + \sum_{l=1}^\infty z_{1,l+1} n^{-2-l}
\]
for Jacobi and
\[
 \frac{x_k}{\beta_n} = \frac{t_k}{n^2} + \sum_{l=1}^\infty z_{1,l+1} n^{-2-l}
\]
for Laguerre, where again $\beta_n$ is the MRS number.

\subsection{The quadrature weights}

The quadrature nodes are the roots of the orthogonal polynomials and are estimated asymptotically using the procedure described above. For the quadrature weights, we recall formula \eqref{Eweights}.
There are several alternative expressions for the weights, yet in the end they should all lead to the same expansion. One advantage of formula \eqref{Eweights} is that it is generally applicable, whereas some other known expressions may be shorter but specific to the classical polynomials.

The simplest approach is to explicitly evaluate the right hand side of \cref{Eweights} using the asymptotic expansion of the orthogonal polynomials. The derivative $p_n'$ can be evaluated by differentiating an expansion for $p_n$ term by term.
The shifted polynomial $p_{n-1}$ has an expansion in powers of $(n-1)$ rather than $n$, but it can also readily be evaluated. Still, we have proceeded to derive explicit expansions for the weights directly.

Thus, we substitute the previously computed expansion \cref{Eformal_root} of the node into \cref{Eweights} and re-expand the entire formula. We take into account the following considerations:
\begin{itemize}
 \item We obtain an asymptotic expansion of $p_{n-1}$ by re-expanding $(n-1)^{-1}$ in inverse powers of $n$.
 \item An expansion for $p_n'$ is obtained by differentiating $p_n$ term by term. However, here, some efficiency can be gained. The expansions typically have the form $p_n(x) \sim E_n(x) \varepsilon(x)$, where $E_n(x)$ is a pre-factor without roots and $\varepsilon(x)$ is the term that vanishes at the roots of $p_n$. Using the fact that $x_k$ is a root of $p_n$, we find that
 \[
  p_n'(x_k) = E_n'(x_k) \varepsilon(x_k) + E_n(x_k) \varepsilon'(x_k) = E_n(x_k) \varepsilon'(x).
 \]
 Thus, we can avoid having to compute the (expansions of the) derivatives of the prefactor $E_n$.
 \item The weight function evaluated at the node, $w(x_k)$, is a common factor in all terms and is factored out once for the final expression.
\end{itemize}

\subsection{Explicit expansions and explicit convolutions}
\label{ss:explicit}

The steps that lead to the expansions of the nodes and weights involve a great number of expansions and re-expansions. Many of these take the form of convolutions: this occurs each time an expansion is applied to a variable that is itself an expansion, or whenever two expansions are multiplied.

In principle, this process is greatly simplified with the aid of computer algebra software. However, any use of such software prohibits an implementation of the derivations of the expansions in a non-symbolic language, though of course the final result can be copied and implemented in any language. Still, even then, the length of the expressions leads to a substantial practical limitation on the number of terms that can be computed.

An alternative is to restrict ourselves to cases where computations can be done analytically. If succesful, this has the advantage of leading to exact expressions in analytic form, for which possibly arbitrary many terms can be generated on the fly. Yet, this strongly restricts the scope of the methodology to cases where analytical derivations are feasible. This rules out any of the modifications of the more general weight functions under consideration.

For these reasons, we have upheld the following two principles in our implementation:
\begin{itemize}
 \item Each expansion of a special function (such as $\sqrt{1+z}$ around $z=0$ or $\frac{1}{n+1}$ in terms of $n^{-k}$) is derived analytically and the resulting formulas are implemented manually. (Hence, we do not rely on the capability of the software to produce series expansions.)
 \item Each convolution is programmed explicitly. Here, too, we do not rely on the algebraic software to recombine, e.g., products of power series into a single power series.
\end{itemize}
These principles make the implementation much more efficient, compared to relying on the symbolic manipulations of a software package.

Having said that, we do compute expansions that contain parameters such as, e.g., $\alpha$ and $\beta$ for the Jacobi-case, $\alpha$ for the Laguerre case, and the $c_n$ and $d_n$ coefficients given by \cref{Ecn} in the case of modified weight functions. If these parameters do not have numeric values, we do currently use the capabilities of a symbolic software package ({\sc Sage}) to manipulate the resulting expressions. When translated into an imperative programming language, all the parameters should have numerical values.

\section{The computation of special functions} \label{SquadrComputSpec}

Some special functions appear in the formulations of the expansions of the polynomials, and the zeros of these functions appear in the expansions for the nodes and weights. Care has to be taken to make sure that these functions can be evaluated at a cost that does not depend on $n$.

\subsection{Zeros of the Bessel and Airy functions}

For the zeros $j_{\alpha,k}$ of the Bessel function $J_\alpha$ of order $\alpha$, we adapt the procedure from \cite[besselroots.m]{chebfun}. All $j_{\alpha,k}$ are initialised with the McMahon expansion \cite[(8)]{McMahon}
\begin{equation}
	j_{\alpha,k} \sim \frac{\pi}{4} (4k + 2\alpha -1) - \frac{4\alpha^2-1}{2\pi(4k + 2\alpha -1)} - \frac{4(4\alpha^2-1)(28\alpha^2 -31)}{24\pi^3(4k + 2\alpha -1)^3} + ..., \quad k \rightarrow \infty. \label{EmcMahon}
\end{equation}
Terms up to $\mathcal{O}(k^{-13})$ are added in a Horner scheme. For $\alpha =0$, the first twenty $j_{0,k}$ are replaced by hard-coded exact double precision values. For other $\alpha \in (-1,5]$, the first six zeros are replaced by Piessens' Chebyshev series approximation \cite[(3 \& 4)]{PiessensBes}, 
\begin{equation}
	j_{\alpha,s} = \left[1 + \delta_{s,1}(\sqrt{\alpha+1} -1) \right] \left(\frac{c_0^{(s)}}{2} + \sum_{k=1}^{N_s} c_k^{(s)} T_k\left[\frac{\alpha-2}{3}\right] \right).
\end{equation}
Here, $\delta_{s,1}$ is the Kronecker delta and $T_k(x)$ is the Chebyshev polynomial of the first kind, $T_k(x) = \cos(k\arccos x)$. The Chebyshev coefficients $c_k^{(s)}$ are listed in \cite{PiessensBes} up to $s=6$ and for $N_s$ such that $c_{N_s+1}^{(s)} < 10^{-12}$. As this is a series approximation, it is not used for $\alpha$ outside $(-1,5]$. For $\alpha > 5$, \cref{EmcMahon} is used, which one can expect not to give accurate results for small $k$. As such, the approximation of $j_{\alpha,k}$ is fast but not very accurate for Bessel parameters outside this range. However, they could be refined by using Newton iterations.

The zeros of the Airy function are approximated by \cite[(9.9.6\&18)]{DLMF} 
\begin{align}
  a_m \sim -t^{2/3}\left(1 + \frac{5}{48}t^{-2} - \frac{5}{36}t^{-4} + \frac{77125}{82944}t^{-6} -\frac{10856875}{6967296} t^{-8} + \frac{162375596875}{334430208} t^{-10} \right),
\end{align}
with $t=3\pi(4m-1)/8$. The first ten values are replaced by hard-coded exact double precision values.

\subsection{Evaluation of the Bessel and Airy functions}

The evaluation of the Bessel function itself is computationally expensive. The problem appears when evaluating $J_{\alpha+1}(j_{\alpha,k}) = -J_{\alpha-1}(j_{\alpha,k})$ in the computation of the weights of the quadrature rule, see \cref{SexplStdGL}. In an implementation where the evaluation of special functions is not available in the language, as is the case for our implementation in Fortran, we have used the following approximations.

For $k \leq 5+10\alpha^2/\pi$, we use the expansion for small arguments \cite[(10.2.2)]{DLMF},
\begin{equation}
 J_\alpha(z) \sim \left(\frac{z}{2}\right)^\alpha \sum_{k=0}^\infty (-1)^k \frac{\left(\frac{z^2}{4}\right)^k}{k! \, \Gamma(\alpha+k+1)}, \quad z \rightarrow 0.
\end{equation}
The series is truncated when the next term is smaller than $10^{-12}$ times the current partial sum.
For larger values of $k$, we use \cite[(10.17.3)]{DLMF}
\begin{align}
	J_\alpha(z) & \sim \sqrt{\frac{2}{\pi z} } \left( \cos\left(z -\frac{\alpha\pi}{2} -\frac{\pi}{4} \right) \sum_{k=0}^\infty (-1)^k \frac{a_{2k}(\alpha)}{z^{2k}}  - \sin\left(z -\frac{\alpha\pi}{2} -\frac{\pi}{4} \right) \sum_{k=0}^\infty (-1)^k \frac{a_{2k+1}(\alpha)}{z^{2k+1}} \right), \\
	a_m(\alpha) &= 2^{-3m} (m!)^{-1} \prod_{n=1}^{m}(4 \alpha^2-(2n-1)^2).
\end{align}
Care is taken to avoid overflow in the factorial function above. The summations are truncated once the terms start to diverge, taking into account that the sine or cosine may be zero. Recall from our comment in  \cref{SquadrPref} that, for large $\alpha$ and/or $\beta$, one may want to consider using different expansions altogether, hence we do not optimize for that regime here.

The expressions for small $z$ for the Airy function and its derivative are implemented similarly. For the derivative of the Airy function evaluated at the first ten zeros of the Airy function, we use hard-coded exact double precision values. The approximation for large negative argument that we use for the other derivatives is \cite[(9.7.10)]{DLMF}
\begin{align}
	\airy'(-z) & \sim \frac{z^{1/4}}{\sqrt{\pi}} \left( \sin\left[\frac{2 z^{3/2}}{3} -\frac{\pi}{4} \right] \left\{ \sum_{k=0}^\infty (-1)^k \frac{v_{2k}}{\left(\frac{2 z^{3/2}}{3}\right)^{2k} } \right\} - \cos\left[\frac{2 z^{3/2}}{3} -\frac{\pi}{4} \right] \left\{ \sum_{k=0}^\infty (-1)^k \frac{v_{2k+1}}{\left(\frac{2 z^{3/2}}{3}\right)^{2k+1} } \right\} \right), \\
	v_m & = \frac{(2m+1)(2m+3)(2m+5) \cdots (6m-1)(6m+1)}{(1-6m) 216^m m!} = \frac{(6m+1) \Gamma(m +5/6) \Gamma(m +1/6)}{(1-6m) 2^{m+1} \pi m!}.
\end{align}

\subsection{Transcendental equation for the bulk in the Laguerre case}
\label{subsect_transcendental}

For the classical Laguerre case, or Laguerre-type case with $Q(x)=x$, equation \cref{Eleading_order} for the leading order term of the nodes becomes the transcendental equation
\begin{equation}
 \label{EtranscStdLag}
 2 \arccos(\sqrt{t}) -2 \sqrt{t-t^2} - p = 0
\end{equation}
with
\begin{equation}
\label{Ep}
 p =\frac{4n-4k+3}{4n+2\alpha+2}.
\end{equation}
We use an iterative Newton procedure, with starting values
\[
 t_{k} \approx \frac{\pi^2}{16}(p-1)^2.
\]
This approximation arises from the series expansion of the left hand side of \cref{EtranscStdLag} near $t=0$, which yields
\[
 p\pi = \pi -4\sqrt{t} +\frac{2}{3} t^{3/2} +\mathcal{O}(t^{5/2}).
\]
Heuristically, we have observed that $6$ iterations are sufficient for double precision accuracy for all $k$ in the bulk region.

The exact form of the transcendental equation results from choices in the Riemann-Hilbert analysis, and could appear differently by selecting another function for splitting the lens. 
However, the rootfinding problem is numerically straightforward, as for each value of $k$ there is a unique simple root in $(0,1)$. This is shown in the following lemma.

\begin{lemma}
The left hand side of equation \cref{EtranscStdLag} has a single simple root in the interval $(0,1)$ for each $k$.
\end{lemma}
\begin{proof}
We denote \cref{EtranscStdLag} as $F(t)=0$. The derivative is
\[
 F'(t) = -(t-t^2)^{-1/2} -\frac{1-2t}{\sqrt{t-t^2}}  = - 2\sqrt{\frac{1-t}{t}},
\]
which is strictly negative, so $F(t)$ is monotonically decreasing on $(0,1)$. The limiting values of $F$ at the endpoints are the maximum
\[
 F(0) = 2\arccos(0) - 2\sqrt{0} - p\pi = \pi \left(1 - p\right),
\]
with $p$ as in \cref{Ep}, and the minimum 
\[
 F(1) = -p \pi = \frac{4n-4k+3}{4n+2\alpha+2}\pi.
\]
Since $k \in [1,n]$, the numerator $4 n - 4 k+3$ lies in $[3,4n-1]$. Moreover, since $\alpha > -1$ the denominator $4 n +2 \alpha +2$ is larger than $4n$. As a result, the maximum is always strictly positive and the minimum is strictly negative, hence there is a single simple zero for $t \in (0,1)$.
\end{proof}

\section{Gauss--Laguerre rules}\label{SLaguerre} 

\subsection{Standard associated Gauss--Laguerre} \label{SexplStdGL}

The weight function in the standard case of associated Gauss--Laguerre is $w(x) = x^\alpha \ee^{-x}$. We have computed asymptotic expansions in terms of inverse powers of $n$. However, the formulae are shorter in terms of inverse powers of $(4n+2\alpha+2)$ and for that reason they are presented that way in this paper. Explicit expressions are listed in Appendix A.

The higher order terms of the expansions contain large integer coefficients with alternating signs, which may form a possible source of cancellation errors. Still, all terms multiply an inverse power of $n$ where the exponent decreases as those integers increase, and therefore any cancellation error will decay accordingly with increasing $n$.
Note that this problem is not specific to asymptotic expansions for Gauss--Laguerre quadrature, it is a rather general problem with expansions.

\subsection{Heuristical choices} \label{SheurLag}

In existing implementations of our approach in \cite{chebfun,FastGaussQuadr} we switch to asymptotics for $n \geq 128$. Below this threshold, we suggest to use the Golub-Welsch algorithm. If very high accuracy is required, weights with magnitude between $10^{-308}$ and $10^{-16}$ may be recomputed using forward recurrence of the orthogonal polynomials. Large values of $\alpha$ may lead to lower accuracy. In our code, we emit a warning when $\alpha^2/n > 1$ and we assume the asymptotics validly apply otherwise.

The simplest heuristic procedure is to sum terms of the asymptotic expansion as long as they decay, and to truncate the sum once the terms start to grow. Experiments indicate this works well. However, in the absence of known error estimates to provide mathematical justification of the truncation, we also embark on a more in depth discussion of truncation of the expansions in different regimes.

First, we limit the number of nodes to compute based on the size of the weights, which may underflow in floating point precision. Indeed, the weight $w_k$ behaves as $w(x_k) = x_k^\alpha \ee^{-Q(x_k)}$.
From \cref{EtranscStdLag} we find, after some calculations, that underflow at a threshold $\epsilon$ happens in the monomial case $Q(x) = q_m x^m$ when
\begin{equation}
  k < \frac{1}{4} -\frac{\alpha}{2} + \frac{(2 n + \alpha +1)}{\pi(2m-1)}\left( \frac{- m A_m \log \epsilon }{2n} \right)^{\frac{1}{2m}},
\end{equation}
where $A_m$ is given by \cref{EAm}. We simplify this expression for $\epsilon = 10^{-308}$ to the heuristical choice
\[
  k \leq \min(n, \exp[1.05 \ee^{\frac{1}{m}}] n^{1-\frac{1}{2m} }).
\]
This agrees with the $\mathcal{O}(\sqrt{n})$ subsampling complexity employed in \cite{TTOGauss} for standard Hermite. In the standard Laguerre case, our heuristic becomes
\begin{equation}
  k \leq \min(n, 17\sqrt{n}).
\end{equation}

A second heuristic concerns the number of terms $T$. Note that the leading order terms for the hard edge and the bulk in \cref{SexplStdGL} only give approximately double precision for $n > 10^8$. One can not simply use all available terms for all $n$, as the asymptotic expansion diverges for fixed $n$ and increasing $T$. Yet, since we adopt a minimal value of $n$, divergence is less of an issue for a certain predetermined accuracy. On the other hand, $T$ may actually decrease as $n$ increases, because fewer terms may be required to reach machine precision. More specifically, when $n$ squares, the number of terms can halve. Thus, $T$ depends inversely logarithmically on $n$. In order to obtain about $10^{-16}$ relative error at $n=128$, we need eight terms. This leads us to our heuristical choice
\begin{equation}
\label{EheurTLag}
 T =\lceil 34/\log(n) \rceil.
\end{equation}

A third heuristic involves the choice of the region, i.e., which expansion to use -- recall the domains shown in \cref{Fregions}. The sizes of the disks near the endpoints are flexible, as mentioned before. We transition between expansions at the values $k^\L$ and $k^\R$. The optimal values of the transition points depend on $T$, $n$, $\alpha$ and, more generally, on the weight function $w(x)$. At least for the standard associated Laguerre weight, the dependency on the parameters appears mild in comparison to the dependency on $n$. Experiments suggest the following scaling:
\begin{equation}
 k^{\L} = \lceil \sqrt{n} \rceil, \qquad \mbox{and} \qquad k^{\R} = \lfloor 0.9n \rfloor. \label{EquadrHeur}
\end{equation}
This corresponds approximately to the points $x_{k^{\L}} \approx \frac{\pi^2}{4}$ and $x_{k^{\R}} \approx 3.6n$.

\subsection{Accuracy} \label{SquadrResExpl}

\begin{figure}[t]
\centering
\includegraphics[width = 0.4\textwidth]{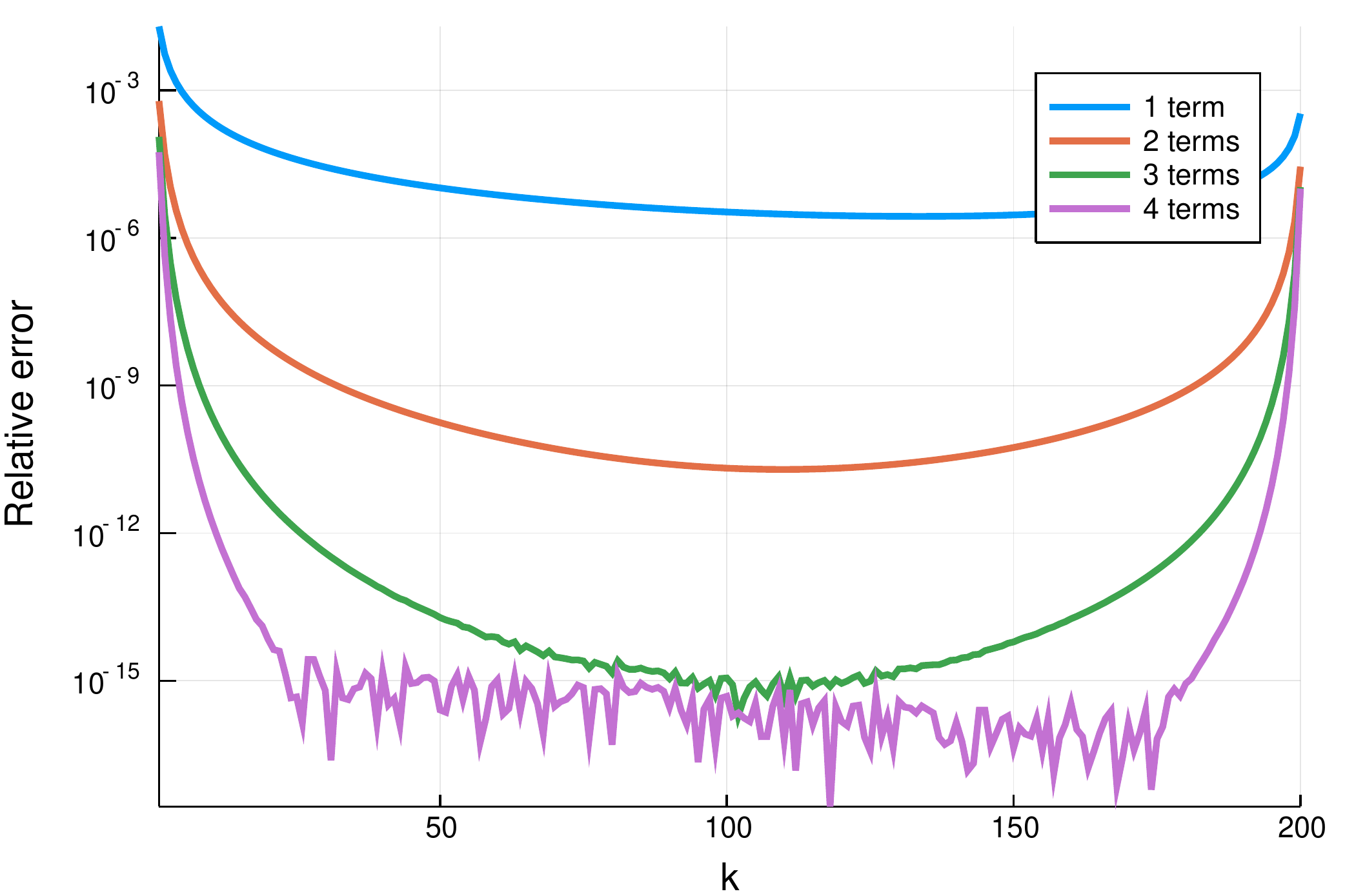}
\caption{Relative error of the nodes in the bulk \cref{EnodeLagBulk} for $n=200$ and $w(x) = x^{0.7} \ee^{-x}$ on $[0,\infty)$. The accuracy of the expression for the bulk deteriorates near the endpoints, hence the need for special boundary asymptotics.}
 \label{FerrNode}
\end{figure}
We validate the heuristics of \cref{SheurLag} by comparing to a reference solution computed using the recurrence relation in high-precision arithmetic. 
In \cref{FerrNode}, we show the relative error for the nodes in the asymptotic expansion \cref{EnodeLagBulk} for the bulk region. The relative error reaches machine precision accuracy in the bulk when using four terms, which corresponds to $\mathcal{O}(n^{-8})$ error. The nodes quickly become less accurate near the left and right endpoints, which is the reason why our heuristics in \cref{SheurLag} switch to another asymptotic expansion there.

In \cref{FerrWeiExpl}, we show the error of the explicit expansions of the weights \cref{EexplLagWeiLeft,EexplLagWeiBulk,EexplLagWeiRight} up to relative order $\mathcal{O}(n^{-2})$, $\mathcal{O}(n^{-4})$, $\mathcal{O}(n^{-6})$ and $\mathcal{O}(n^{-8})$ combined with the heuristic \cref{EquadrHeur} at $n=200$. As expected, increasing the number of terms in those expansions decreases the error. 

The expansions of the weights match closely when we switch from the left disk to the bulk near $k^{\L} = 15$. A steep jump in the relative error is seen near $k^{\R} = 180$ when we switch to Airy asymptotics. This is because the expansion of the weight in the latter region \cref{EexplLagWeiRight} is only accurate up to $\mathcal{O}(n^{-2/3})$. However, the first weight in this regime has size $w_{180} \approx 9 \times 10^{-222}$, hence the absolute error is very small. It is a remarkable property of asymptotic expansions that such a small weight can still be computed with some relative accuracy.

\begin{figure}[t]
\centering
\includegraphics[width = 0.4\textwidth]{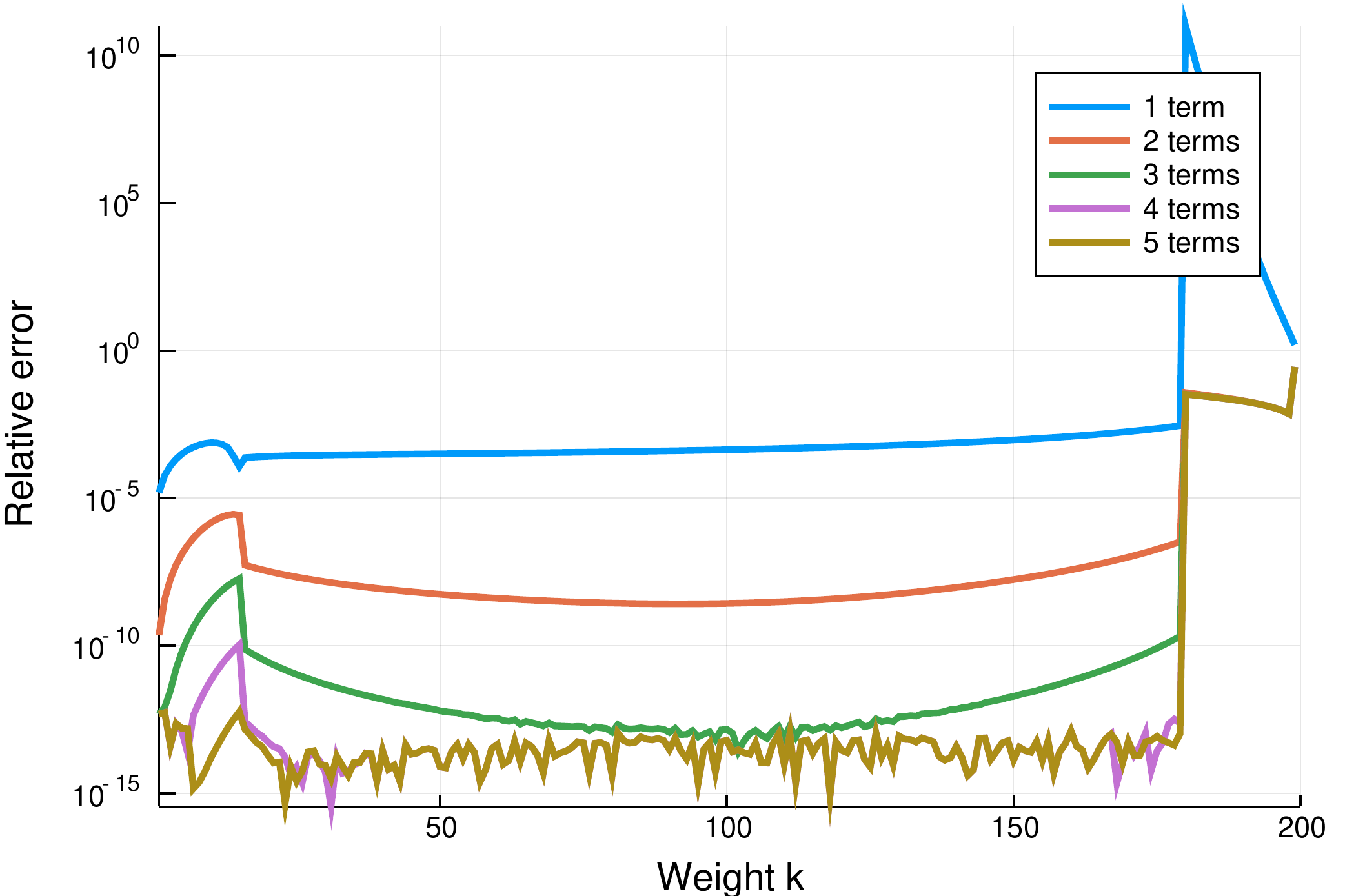}
\caption{Relative error of the weights for $n=200$ and $w(x) = x^{0.7} \ee^{-x}$ on $[0,\infty)$. At $k = 0.9 n = 180$ we switch to Airy asymptotics with fewer terms and seemingly less accuracy (visible as the spike in the figure). Yet, because $w_{180} < 10^{-220}$ in this case, the absolute error is still extremely small.}
\label{FerrWeiExpl}
\end{figure}

\subsection{Modified Gauss--Laguerre} \label{SexplGenGL}

We do not pursue all expansions in the full generality of the modification of the weight function \cref{Elaguerre_weight}, i.e., for all possible analytic functions $Q(x)$. In \cite{laguerre} we considered asymptotic expansions for the polynomials when $Q(x)$ is a monomial, a general polynomial of degree $m$, or a more general analytic function on the positive halfline for which the orthogonal polynomials exist. In principle, these expansions could be used to generate asymptotic expansions for the corresponding Gaussian quadrature rules. However, the level of complexity of the computation increases.

We briefly elaborate on some of the issues. First, the procedure for general polynomial $Q(x)$ yields fractional powers of $n$, due to expanding the MRS number $\beta_n$ in fractional powers of $n$ as shown in \cite[\S 3]{laguerre}. This increases the complexity of obtaining full asymptotic expansions quite substantially, as the computation of not one but several terms is required to increase the order by $1/n$.

Next, general functions $Q(x)$ require the calculation of contour integrals. Recall that a weight of the form $w(x)=x^\alpha \ee^{-Q(x)}$ leads to the coefficients $c_n$ and $d_n$ that were given by contour integrals \cref{Ecn2}--\cref{Edn2}. These contour integrals have to be computed, either analytically or numerically, since the coefficients appear in the expansions. One advantage of the monomial case $Q(x)=x^m$ is that the contour integrals \cref{Ecn2}--\cref{Edn2} can be evaluated analytically for all $m$. Other examples where these contour integrals can be found analytically are given in \cite[\S3.4]{laguerre} and \cite[(2.19)]{opsomer2018phd}. Similar contour integrals also appear in the formula for the phase function in the leading order term of the expansion of the orthogonal polynomial. Expressions are given in \cite[\S 6]{laguerre}, and they are the analogue of \cref{Elambda} for the Jacobi case.

Finally, since we have to numerically solve for the roots of the leading order term, recall \cref{Eleading_order}, with the leading order term for general $Q(x)$ itself involves requiring the evaluation of contour integrals, producing expansions may be quite time-consuming.

For these reasons, the expansions for modified Laguerre that we list explicitly in this paper are somehwat limited. Still, for a specific choice of weight function, the code may be used to produce expansions.

\paragraph{Hard edge.} 

The expansion of Laguerre-type nodes near the left endpoint is 
\begin{align}
	x_k & = \frac{16 j_{\alpha,k}^2 \beta_n}{d_0^2 (4n+2\alpha+2)} \bigg(1 + \frac{4[\alpha+1] [d_0 -2]}{d_0 (4n +2\alpha +2)} \label{EnodeLagLeftGen} \\
	& + \frac{4}{3c_0^2 d_0^3 (4n +2\alpha +2)^2} \bigg[ 9(\alpha^2 + 2\alpha + 1)c_0^2 d_0^3 + 2(22\alpha^2 + 36\alpha + 17)c_0^2 d_0 -(4\alpha^2 - 1) c_0^2 d_1  \\ 
	& - 3(12[\alpha^2 + 2\alpha + 1]c_0^2 - 2[2\alpha^2 + 4\alpha+ 1]c_0 + c_1)d_0^2 + 4(c_0^2 d_0 - 2c_0^2 d_1)j_{\alpha,k}^2 \bigg] \\
	& + \frac{16}{3 c_0^4 d_0^4 (4n +2\alpha +2)^3} \Big[ 6(\alpha^3 + 3\alpha^2 + 3\alpha + 1)c_0^4 d_0^4 - 3(28\alpha^3 + 60\alpha^2 +45\alpha + 13) c_0^4 d_0 \\ 
	& - 3\big(12\{\alpha^3 + 3\alpha^2 + 3\alpha + 1\}c_0^4 -4(2\alpha^3 + 6\alpha^2 + 5\alpha + 1)c_0^3 \\
	& + 2(2\alpha^3 + 6\alpha^2 + 3\alpha -1)c_0^2 - 6(\alpha + 1)c_1^2 + 5(\alpha + 1)c_0 c_2 + \{2(\alpha + 1)c_0^2 + (4\alpha^3+ 12\alpha^2 + 5\alpha - 3)c_0\} c_1\big)d_0^3 \\
	& + \big(4[22\alpha^3 + 58\alpha^2 + 53\alpha +17]c_0^4 - 3[16\alpha^3 + 40\alpha^2 + 29\alpha + 5]c_0^3 + 9[\alpha +1]c_0^2c_1\big)d_0^2 \\
	& + 8\big([\alpha + 1]c_0^4d_0^2 - 3(\alpha + 1)c_0^4 d_0 - [2(\alpha +1)c_0^4 d_0 - 5(\alpha + 1)c_0^4]d_1\big)j_{\alpha,k}^2 \\
	& - \big(2[4\alpha^3 + 4\alpha^2 - \alpha -1]c_0^4 d_0 - 5[4\alpha^3 + 4\alpha^2 - \alpha - 1]c_0^4\big)d_1 \Big] + \mathcal{O}(n^{-4}). 
\end{align}

\paragraph{Bulk.} 
Based on the asymptotic expansions in \cite[\S3.4]{laguerre}, equation \cref{Eleading_order} for the leading order specializes to
\begin{equation}
\pi(4k - 4n - 3) +(\alpha+1)2\arccos(2t-1) +\frac{n}{4i}  \int_1^t \frac{\sqrt{y-1}}{\sqrt{y}} \left[\frac{1}{2\pi i}\oint_{\Gamma_y} \frac{\beta_n \sqrt{x} Q^\prime(x) \dint x}{n \sqrt{x-1}(x-y)}\right] \dint y = 0.
\end{equation}
The contour $\Gamma_y$ should enclose the interval $[0,1]$ and the point $y$. 

Unfortunately, this expression does not simplify even in the monomial case.

\begin{figure}[t]
\centering
\includegraphics[width=0.8\hsize]{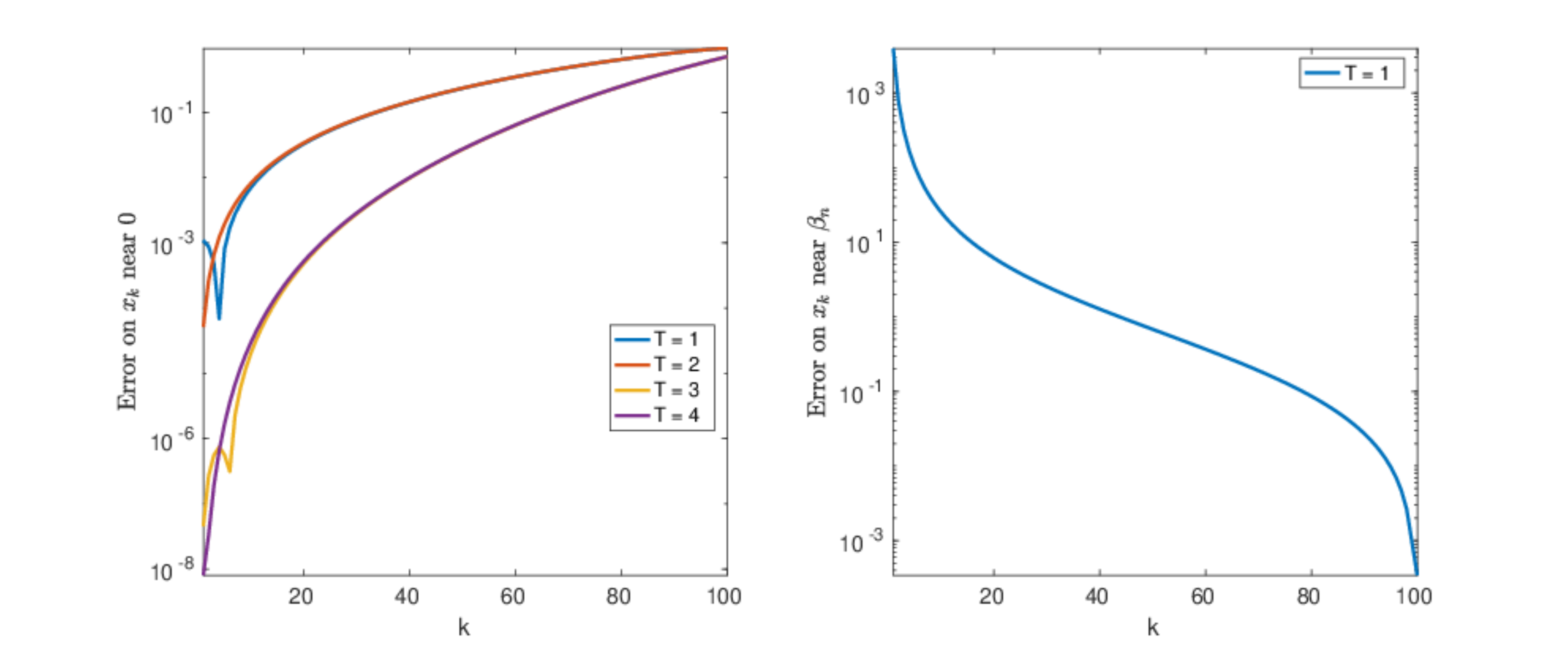}
\caption{Relative error of the explicit expansions of the nodes \cref{EnodeLagLeftGen} near the left endpoint (left panel) and \cref{ExkSoft} near the soft edge (right panel) for the weight function $w(x) = \ee^{-\ee^x}$ on $[0,\infty)$ and $n=100$.}
\label{FGL0espErrExpa}
\end{figure}

\paragraph{Soft edge.} 
We only provide the leading order term near the soft edge, as the corresponding weights underflow even for moderate $n$. 
We have
\begin{equation}
 x_k \sim \beta_n \left[ 1+ \left( \frac{2}{n c_0} \right)^{2/3} a_{n-k+1}  + o(n^{-2/3}) \right] \label{ExkSoft} 
\end{equation}
and
\begin{align}
w_k^{-1} & \sim 2 \beta_n^{n-1}\beta_{n-1}^{1-n-\alpha-1}\exp\left[ n l_n/2 -(n-1) l_{n-1}/2 \right] \left(\frac{\beta_n-\beta_{n-1}}{\beta_{n-1}}  + \left(\frac{2}{nc_0}\right)^{2/3} a_{n-k+1} \right)^{-1/4} \\
	& \exp[Q(x_k)] Ai'(a_{n-k+1})^2 \left(c_0/2\right)^{3/2} n^{1/2} a_{n-k+1}^{1/4}.
\end{align}
Here, recall that $a_m$ are the roots of the Airy function. The constant $l_n$ is given for general Laguerre-type weight functions in \cite[\S3.2]{laguerre} and also involves an integral to be evaluated numerically.

In \cref{FGL0espErrExpa}, we can see that the expansions near the hard and the soft edge give good approximations. Moreover, increasing the number of terms $T$ near the hard edge decreases the error, with results shown using $n=100$ for the nodes of a Gauss--Laguerre-type quadrature rule with respect to the doubly-exponential weight function 
\begin{equation}
	w(x) = \exp(-\ee^x).
\end{equation}
In that case we have
\[
 \beta_n \sim \log(n) - \log(\log(8\pi n^2))/2 + \log(8\pi)/2.
\]

\section{Gauss--Jacobi rules}\label{SJacobi} 

\subsection{Standard Gauss--Jacobi} \label{SexplStdGJ}

The polynomials exhibit Bessel-like behaviour near both hard edges at $+1$ and $-1$. Similar to the Laguerre case, the expansions were computed in terms of inverse powers of $n$, but the results are presented in shorter form using inverse powers of $(2n+\alpha+\beta+1)$ (in part following the notation of \cite{pitta}). The expansion for the nodes $x_k$ near $x=-1$ are given in~\eqref{EexkStdJacBes} up to $\mathcal{O}(n^{-12})$. The corresponding weights $w_k$ are given up to relative order $\mathcal{O}(n^{-8})$, relative to the size of $w(x_k)$, in \eqref{EewkStdJacBes}. Near the right endpoint at $x=+1$, we can simply interchange $\alpha$ and $\beta$ in these expressions, and multiply $x_{n+1-k}$ by $-1$ for $k=1,\ldots,n$.

For the nodes and weights in the bulk, we first have to find the roots of the leading order term of the polynomials. In the modified Jacobi case the generic equation \cref{Eleading_order} specializes to \cref{EtkBulkJacModif} further on. In the current classical case, there is no contour integral and the roots are obtained with the explicit expression
\begin{equation}
\label{EtkBulkJacClassical}
 t_k = \cos\left( \pi  \frac{4n-4k + 2\alpha +3}{4n+2\alpha+2\beta+2}\right).
\end{equation}
The expansion of the nodes up to $\mathcal{O}(n^{-10})$ and weights up to (relative order) $\mathcal{O}(n^{-8})$ are given in~\eqref{EexkStdJacBulk} and~\eqref{EewkStdJacBulk} respectively.

\subsection{Heuristical choices}
\label{SheurJac}

We are faced with similar choices as in the Laguerre case in \cref{SheurLag}, namely the choice of the number of terms $T$ in the expansion and the choice of the expansion to use. Due to the difference in nature with Gauss-Laguerre and because we have fewer terms available in the current case, we adjust our assumption for $n$: we suggest to use the recurrence relation up to about $n \leq 300$ for double precision accuracy. We assume that the asymptotics validly apply above this limit if in addition $(\alpha^2 + \beta^2) < n$. Since there is no underflow associated with the weights of Gauss--Jacobi at moderate $n$, we compute all nodes and weights.

Based on our aim of achieving machine precision in double precision accuracy, we heuristically choose
\[
T = \lceil 50 / \log(n) \rceil.
\]
This is slightly higher than the corresponding choice \cref{EheurTLag} in the Laguerre case.

For the switch between the asymptotic regimes, we have settled on the choices
\begin{equation}
 k^{\L} = \lceil \sqrt{n} \rceil, \qquad \mbox{and} \qquad k^{\R} = n - \lceil \sqrt{n} \rceil.
\end{equation}
This corresponds to switching approximately at the points $x_{k^{\L}} = \frac{\pi^2}{2n} -1$ and $x_{k^{\R}} = 1 - \frac{\pi^2}{2n}$.

\subsection{Accuracy} 

In \cref{Fjacobi_accuracy}, we show the maximal absolute or relative error over all nodes or weights at each integer value of $n$ between $1$ and $400$. The asymptotic expansions are compared to a computation of the corresponding Gaussian quadrature rule using the recurrence relation in higher precision arithmetic. This allows to assess the accuracy to all digits.

\setlength{\figurewidth}{9cm}
\setlength{\figureheight}{7cm}
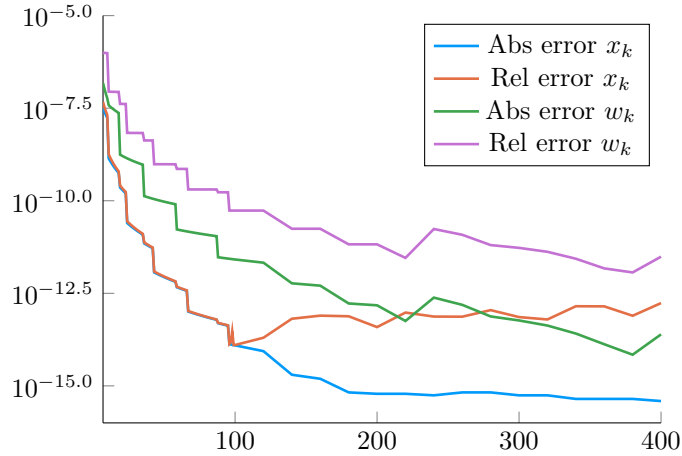
\begin{figure}
 \centering
 \begin{small}
\input{jacobi_accuracy.tikz}
\end{small}
\caption{Maximal absolute and relative error of the points and weights of all Jacobi quadrature rules for $n$ up to $400$ for the weight function $w(x) = (1-x)^{0.42} (1+x)^{-1/\sqrt{5}}$.}
\label{Fjacobi_accuracy}
\end{figure}

\subsection{Modified Gauss--Jacobi} \label{SexplGenGJ}

For the sake of brevity, we only give the first few terms and indicate how many were computed in our implementation. Recall that these expressions are symbolic because they contain the parameters $\alpha$ and $\beta$. When numeric values for $\alpha$ and $\beta$ are supplied, more terms can be readily computed with modest effort. In contrast to the case of modified Laguerre, the results here are quite complete.

\paragraph{Hard edge at $x=+1$.} 
The expansions near the left endpoint $x=-1$ are 
\begin{align}
	x_k & \sim -1 +\frac{2j_{\beta,k}^2}{ (2n+\alpha+\beta+1-d_0)^2} + \frac{-2j_{\beta,k}^2}{3 (2n+\alpha+\beta+1-d_0)^4} \left[j_{\beta,k}^2 - 3 \alpha ^2 - \beta^2 + 1 \right] \label{EexkGenJacBes} \\
	& + \frac{-j_{\beta,k}^2}{6 (2n+\alpha+\beta+1-d_0)^5}\left[ \vphantom{\frac{-j_{\beta,k}^2}{6 (2n+\alpha+\beta+1-d_0)^5}} 16 ( d_0  - 3  d_1 ) j_{\beta,k}^4 + 3 (4 \alpha ^2 -1)  c_0  + (12 \alpha ^2 + 8 \beta ^2 - 5)  d_0  - 6 (4 \beta ^2 - 1)  d_1 \right] \\
	& + \ldots + \mathcal{O}(n^{-9}).
\end{align}
The corresponding weights are asymptotically
\begin{align}
	\frac{w_k}{w(x_k)} & \sim \frac{8}{ J_{\beta-1}^2 (j_{\beta,k}) [2n+\alpha+\beta+1-d_0]^2} \label{EewkGenJacBes} \\
	&+ \frac{8}{3 J_{\beta-1}^2 (j_{\beta,k}) [2n+\alpha+\beta+1-d_0]^4} \left[ \vphantom{\frac{-j_{\beta,k}^2}{6 (2n+\alpha+\beta+1-d_0)^5}} 3 \alpha ^2 + \beta^2 - 1 -2j_{\beta,k}^2 \right]  \\
	& - \frac{2\left[ \vphantom{\frac{-j_{\beta,k}^2}{6 (2n+\alpha+\beta+1-d_0)^5}} 32(d_0 - 3 d_1)j_{\beta,k}^2 + 3(4\alpha^2 - 1) c_0 + (12\alpha^2 + 8\beta^2 - 5)d_0 - 6(4\beta^2 - 1)d_1 \right]}{3 J_{\beta-1}^2 (j_{\beta,k}) [2n+\alpha+\beta+1 -d_0]^5}  \\
	& + \ldots + \mathcal{O}(n^{-8}). 
\end{align}

\paragraph{Hard edge at $x=+1$.} Near the right endpoint at $x=+1$, we can simply interchange $\alpha$ and $\beta$ in the expressions above, and multiply $x_{n+1-k}$ by $-1$ for $k=1,\ldots,n$. This implies that one has to substitute $h(-x)$ for $h(x)$ when recomputing the expansions coefficients $c_n$ and $d_n$.

\paragraph{Bulk.}
For the expansions in the bulk, equation \cref{Eleading_order} for the parameter $t_k$ becomes 
\begin{align}
	\pi  \frac{4k + 2\alpha +3}{4n+2\alpha+2\beta+2}  & = \arccos(t_k) + \frac{\sqrt{1-t_k^2}}{2n+\alpha+\beta+1} \frac{1}{2\pi i} \oint_\gamma \frac{\log h(\zeta) \dint \zeta}{\sqrt{\zeta^2-1} (\zeta -t_k)}. \label{EtkBulkJacModif}
\end{align}
This corresponds to the expression in \cite{pitta}, with a minor modification of the denominator and the addition of a contour integral for our generalised case.

One needs to compute the series expansion coefficients of the latter:
\begin{align}
	\frac{1}{2\pi i} \oint_\gamma \frac{\log h(\zeta) \dint \zeta}{\sqrt{\zeta^2-1} (\zeta -z)} & \sim \sum_{i=0}^\infty h_i (z-t_k)^i. \label{EcintloghSeries} 
\end{align}
In the absence of an analytical approach for the evaluation of these integrals, they have to be computed numerically for each $k$. Examples of cases where they can be evaluated analytically are given in \cite[\S 6.1]{jacobi}.

Still, in terms of the coefficients $h_i$, the expansion for the nodes in the bulk region is, up to $\mathcal{O}(n^{-5})$:
\begin{align}
	x_k & \sim t_k + \frac{2\alpha^2 - 2\beta^2 + (2\alpha^2 + 2\beta^2 - 1)t_k}{2[2n + \alpha + \beta + 1 + h_0]^2} - \frac{1}{4[2n + \alpha + \beta + 1 + h_0]^3} \left(\vphantom{\frac{1^2}{x^2}} 2(2\alpha^2 + 2\beta^2 - 1)h_1 t_k^3 \right. \label{EexkGenJacBulk} \\ 
	& + 2\left[(2\alpha^2 + 2\beta^2 - 1)h_0 + 2(\alpha^2 - \beta^2)h_1\right]t_k^2  + (4\alpha^2 - 1)c_0 + (4\beta^2 - 1)d_0 + 8(\alpha^2 - \beta^2)h_0 - 4(\alpha^2 - \beta^2)h_1 \\ 
	& \left. + \left[(4\alpha^2 - 1)c_0 - (4\beta^2 - 1)d_0 + 4(3\alpha^2 + \beta^2 - 1)h_0 - 2(2\alpha^2 + 2\beta^2 - 1)h_1 \right]t_k\vphantom{\frac{1^2}{x^2}} \right) + \ldots + \mathcal{O}(n^{-5}).
\end{align}
The corresponding weights are, up to $\mathcal{O}(n^{-4})$ relative error (relative to $w(x_k)$):
\begin{align}
	\frac{w_k}{w(x_k)} & \sim \frac{\pi\sqrt{1-t_k^2}}{2n + \alpha + \beta + 1} \bigg[2 - \frac{2h_1(1-t_k^2) -2h_0 t_k}{2n + \alpha + \beta + 1} + \frac{1}{(2n + \alpha + \beta + 1)^2} \big( 2h_1^2 t_k^4 + 4 h_0 h_1 t_k^3 - 4 h_0 h_1 t_k \\ 
	& + 2(h_0^2 - 2h_1^2)t_k^2 + 2\alpha^2 + 2\beta^2 + 2h_1^2 - 1\big) + \ldots + \mathcal{O}(n^{-4}) \bigg]. \label{EewkGenJacBulk} 
\end{align}

\begin{figure}[t]
\centering
\includegraphics[width=0.8\hsize]{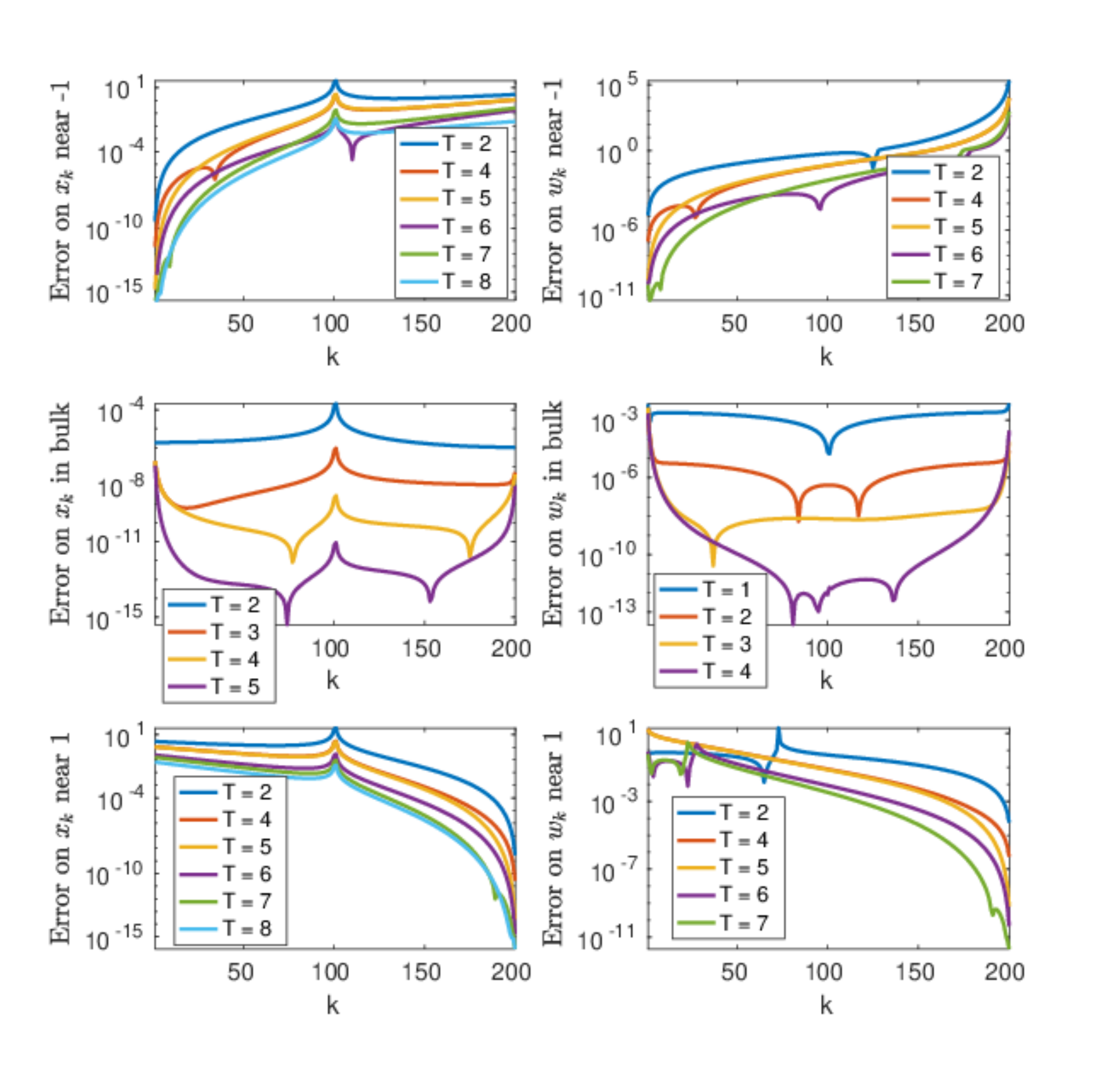}
\caption{Relative error of the explicit expansions of the nodes \cref{EexkGenJacBes} and weights \cref{EewkGenJacBes} near the left endpoint (top), and the nodes \cref{EexkGenJacBulk} and weights \cref{EewkGenJacBulk} in the bulk (middle) for the weight function $w(x) = (1-x)^{1/\sqrt{3}} (1+x)^{-1/\pi} \ee^x$ at $n=200$ for a varying number of terms $T$. 
The expressions for the left endpoint are reused for those near the right endpoint (bottom) by interchanging $\alpha$ and $\beta$ and computing $c_n$ and $d_n$ with $h(x) = \ee^{-x}$.}
\label{FGJsq3piexpoErrExpa}
\end{figure}

We illustrate the accuracy of these expansions throughout the domain $[-1,1]$ in \cref{FGJsq3piexpoErrExpa} for $n=200$ for the weight function 
\[
  w(x) = (1-x)^{1/\sqrt{3}} (1+x)^{-1/\pi} \ee^x.
\] 
The results show that the expansions improve with $T$ and that their accuracy is very high in the regions where they are valid.

\section{Gauss--Hermite rules} \label{SquadrGH} 

The standard Gauss--Hermite polynomials are related to associated Gauss--Laguerre polynomials with $\alpha = \mp 1/2$ \cite[18.7.19-20]{DLMF}. Written in terms of the orthonormal polynomials, we have the relations
\begin{equation}
  p_{2n}^{(\qh)}(x) = p_n^{(\ql, -\frac12)}(x^2), \qquad p_{2n+1}^{(\qh)}(x) = x p_n^{(\ql,\frac12)}(x^2). \label{EhermNormLag}
\end{equation}
This generalises to other functions $Q(x)$ with $\alpha = \mp 1/2$ as was shown in \cite[\S 7.2]{laguerre}. Using the change of variables $x = t^2$, we have
\begin{equation}
\label{ELagHermite}
 \int_0^\infty x^j \, p_n(x) \, x^{-\frac12} \ee^{-Q(x)} \, {\rm d}x = 2 \int_0^\infty t^{2j} \, p_n(t^2) \ee^{-Q(t^2)} \, {\rm d}t = \int_{-\infty}^\infty t^{2j} \, p_n(t^2) \ee^{-Q(t^2)} \, {\rm d}t = 0, \quad j=0,\ldots,2n-1.
\end{equation}
Since the odd moments of $p_n(t^2)$ vanish by symmetry, it follows that $p_n(t^2)$ is the Hermite-type polynomial of degree $2n$ with respect to the weight function $\ee^{-Q(t^2)}$ on the real line.

Similarly, for the odd-degree polynomials we have
\begin{equation}
\label{ELagHermiteOdd}
 \int_0^\infty x^j \, p_n(x) \, x^{\frac12} \ee^{-Q(x)} \, {\rm d}x = \int_{-\infty}^\infty t^{2j+1} \, t\, p_n(t^2) \ee^{-Q(t^2)} \, {\rm d}t = 0, \quad j=0,\ldots,2n-1.
\end{equation}
In this case, the even moments vanish due to the odd symmetry of $p_n(t^2)$.

It remains to determine the connection between the corresponding Gaussian quadrature nodes and weights.

\begin{lemma}
Let $x_k^{(\ql,-\frac12)}$ and $w_k^{(\ql,-\frac12)}$, $k=1,\ldots,n$, be the nodes and weights for the modified Gauss--Laguerre rule with weight function $x^{-\frac12} \ee^{-Q(x)}$. Then the nodes and weights of the modified Gauss--Hermite rule with weight function $\ee^{-Q(x^2)}$ on $(-\infty,\infty)$ are, with $k=1,\ldots,n$,
\begin{equation}
	x_{n+k}^{(\qh)} = \sqrt{x_k^{(\ql, -\frac12)}}, \quad w_{n+k}^{(\qh)} = \frac{w_{k}^{(\ql,-\frac12)}}{2}, \quad x_{n+1-k}^{(\qh)} = -\sqrt{x_k^{(\ql, -\frac12)}}, \quad w_{n+1-k}^{(\qh)} = \frac{w_{k}^{(\ql,-\frac12)}}{2}.
\end{equation}
\end{lemma}
\begin{proof}
The result for the nodes follows immediately from the change of variables in \cref{ELagHermite}. The results for the weights follow by applying the same change of variables to the moment conditions of the Laguerre-type quadrature rule. Since
\[
 \int_0^\infty x^j x^{-\frac12} \ee^{-Q(x)} \, {\rm d}x = \sum_{k=1}^n w_k^{(L,-\frac12)} \left(x_k^{(L,-\frac12)}\right)^j, \qquad j=0,\ldots,2n-1,
\]
we have
\[
 \int_{-\infty}^\infty t^{2j} \ee^{-Q(t^2)} \, {\rm d}t = \sum_{k=1}^n w_k^{(L,-\frac12)} \left(x_k^{(L,-\frac12)}\right)^{j} = \sum_{k=1}^{2n} w_k^{(H)} \left(x_k^{(H)}\right)^{2j}, \qquad j=0,\ldots,2n-1.
\]
The moment conditions for the odd monomials $t^{2j+1}$, $j=0,\ldots,2n-1$, follow from the odd symmetry.
\end{proof}

The result for Gauss--Hermite type rules with an odd number of points is similar, with some complications only for the extra node at the origin.

\begin{lemma}
The nodes and weights of a $2n+1$ point modified Gauss--Hermite quadrature rule are given in terms of the modified Gauss--Laguerre quadrature rule with $\alpha=+\frac12$ as, with $k=1,\ldots,n$,
\begin{align}
	x_{n+1+k}^{(\qh)} & = \sqrt{x_k^{(\ql, \frac12)}}, & & w_{n+1+k}^{(\qh)} = \frac{w_{k}^{(\ql,\frac12)}}{2x_k^{(\ql,\frac12)} }, \\
	x_{n+1}^{(\qh)} & = 0, & & w_{n+1}^{(\qh)} = \int_{-\infty}^\infty \ee^{-Q(t^2)} \, {\rm d}t - \sum_{k=1}^n \left[w_{n+1+k}^{(H)} + w_{n+1-k}^{(H)}\right], \\ 
	x_{n+1-k}^{(\qh)} & = -\sqrt{x_k^{(\ql, 1/2)}}, & & w_{n+1-k}^{(\qh)} = \frac{w_{k}^{(\ql,\frac12)}}{2x_k^{(\ql,\frac12)} }.
\end{align}
\end{lemma}
\begin{proof}
The moment conditions for the modified Laguerre rule are
\[
 \int_0^\infty x^j x^{\frac12} \ee^{-Q(x)} \, {\rm d}x = \int_{-\infty}^\infty t^{2j+2} \, \ee^{-Q(t^2)} \, {\rm d}t= \sum_{k=1}^n w_k^{(L,\frac12)} \left(x_k^{(L,\frac12)}\right)^j, \qquad j=0,\ldots,2n-1.
\]
The modified Gauss--Hermite rule should satisfy
\[
 \int_{-\infty}^\infty t^m \, \ee^{-Q(t^2)} \, {\rm d}t = \sum_{k=1}^{2n+1} w_k^{(H)} \left(x_k^{(H)}\right)^m, \qquad m=0,\ldots,4n+1.
\]
The odd moments are satisfied by symmetry. The even moments are satisfied for $m = 2j+2 > 0$ using the exactness conditions for Laguerre above, noting that $\left(x_{n+1}^{(H)}\right)^m=0$ since $m > 0$. The moment condition for $m=0$ is satisfied precisly by the weight formulated in the lemma.
\end{proof}

There is no known general formula for the middle weight $w_{n+1}^{(\qh)}$. However, in case of the standard Gauss--Hermite rule once can find from \cite[\S 3.5 \& 18.6]{Olver:2010:NHMF} that
\[
 w_{n+1}^{(\qh)} = \frac{\pi \Gamma(n+1)}{(2n+1)\Gamma(n + \frac12)}.
\]
For large $n$, the ratio of Gamma functions can be approximated much like in \cite[\S 3.2.3]{HaleTownsend} as
\begin{align}
	\frac{\Gamma(n+1)}{\sqrt{n}\Gamma(n+1/2)} & \sim 1 + \frac{1}{8n} + \frac{1}{128 n^2} -\frac{5}{1024 n^3} -\frac{21}{32768 n^4} + \frac{399}{262144 n^5} + \frac{869}{4194304 n^6} \\
	& - \frac{39325}{33554432 n^7} - \frac{334477}{2147483648 n^8} + \frac{28717403}{17179869184 n^9} + \frac{59697183}{274877906944 n^{10}}, \quad n \rightarrow \infty.
\end{align}

In order to construct the Gauss-Hermite rule, one simply computes the Gauss-Laguerre rules from \cref{SexplStdGL} with approximately half the number of points and then exploits these relations.

The connection to Laguerre-type polynomials seemingly obviates the need to repeat the asymptotic analysis for Hermite-type polynomials, at least for even symmetric weight functions of the form $\ee^{-Q(x^2)}$. It may still be worthwhile to repeat the analysis for symmetric weight functions, since the Hermite-type polynomials do not have a hard edge at $x=0$. Hence, the computation of their expansion is simpler and it becomes more tractable to compute more terms. For example, for symmetric weight functions and $\alpha = \pm 1/2$ the $U_{k,m}^{\L}$ matrices in \cite{Vanlessen,laguerre} are zero. The analysis may also be useful for non-symmetric weight functions and for algebraic singularities near zero, i.e. $\alpha \neq \pm 1/2$.

\section*{Acknowledgements}
The authors gratefully acknowledge financial support from FWO (Fonds Wetenschappelijk Onderzoek, Research Foundation - Flanders, Belgium), through FWO research projects G.0617.10, G.0641.11 and G.A004.14. The authors thank Ronald Cools, Vincent Copp{\'e}, Arno Kuijlaars, Alex Townsend and Marcus Webb for their suggestions to improve the manuscript and their help in the numerical examples.

\appendix

\section{Expansions for classical Gaussian quadrature rules}
\label{app:explicit}

We include explicit expressions for the asymptotic expansions of nodes and weights of Gauss--Laguerre and Gauss--Jacobi quadrature rules. Expressions for Gauss--Legendre can be found in \cite{BogaertIterationFree}. Earlier results on the classical rules include, among others, \cite{hahn,baraGatt,gil2016,gatteschi1985zeros,tricBess,tricBulk,gatt,alfLag,GilSeguraLaguerre,GST,GilSeguraQuadr}.

\subsection{Gauss--Laguerre}
\label{app:laguerre}

Asymptotic expansions for the nodes and weights of Gauss--Laguerre quadrature rules were described concurrently in \cite{GST,opsomer2018phd}. Here, we recall the explicit expansions as they are derived and presented in \cite{opsomer2018phd}.

\paragraph{Hard edge.} Denote by $j_{\alpha,k}$ the $k$-th root of the Bessel function of order $\alpha$. An explicit expansion for the nodes near the hard edge at $x=0$ is:
\begin{align} 
	x_k & = \frac{j_{\alpha,k}^2}{4n+2\alpha+2} + \frac{j_{\alpha,k}^2\left(j_{\alpha,k}^2  +2 \alpha^2 -2\right)}{3(4n+2\alpha+2)^3} \label{EnodeLagLeft}  \\ 
	&  + \frac{j_{\alpha,k}^2\left( 11 j_{\alpha,k}^4  + 3 j_{\alpha,k}^2 (11\alpha^2 -19) + 46\alpha^4 - 140\alpha^2 + 94 \right)}{45(4n+2\alpha+2)^5}  \\ 
	& + \frac{ j_{\alpha,k}^2}{3^4 35 (4n+2\alpha+2)^7}\left[657 j_{\alpha,k}^6 + 36 j_{\alpha,k}^4 (73\alpha^2 - 181) \right.  \\
	& \left. + 2 j_{\alpha,k}^2 (2459\alpha^4 - 10750\alpha^2 + 14051) + 4 (1493\alpha^6 - 9303\alpha^4 + 19887\alpha^2 - 12077) \right] \\
	& + \frac{ j_{\alpha,k}^2}{3^5 5^2 7 (4n+2\alpha+2)^9}\left[ 10644 j_{\alpha,k}^8 + 60 j_{\alpha,k}^6 (887\alpha^2 - 2879) \right. \\
	& + j_{\alpha,k}^4 (125671\alpha^4 -729422\alpha^2 + 1456807) \\
	& + 3 j_{\alpha,k}^2 (63299\alpha^6 - 507801\alpha^4 + 1678761\alpha^2 -2201939)  \\
	& \left. + 2(107959\alpha^8 - 1146220\alpha^6 + 5095482\alpha^4 -10087180\alpha^2 + 6029959)\right]+ \mathcal{O}(n^{-11}).
\end{align}
The inequality $x_{k} > j_{\alpha,k}^2(4n+2\alpha+2)^{-1}$ \cite[(18.16.10)]{DLMF} already provides an $\mathcal{O}(n^{-2})$ relative error and shows that the remainder has to be positive. The first two terms are exactly those found by Tricomi \cite[(37)]{tricBess}, where the error bound was reported to be $\mathcal{O}(n^{-4})$, while \cite{gatt} reported a stricter $\mathcal{O}(n^{-5})$.

The corresponding expansion for the weights is
\begin{align}
	w_k & = \frac{4 x_k^\alpha \, \ee^{-x_k}}{J_{\alpha-1}^2(j_{\alpha,k})(4n + 2\alpha +2)} \left(1 +  \frac{2(\alpha^2 + j_{\alpha,k}^2 -1)}{3(4n + 2\alpha +2)^2} \right. \label{EexplLagWeiLeft} \\ 
	&  + \frac{1}{45 (4n + 2\alpha +2)^4} \left[ 46\alpha^4 + 33 j_{\alpha,k}^4 +6 j_{\alpha,k}^2 (11\alpha^2 -19) -140\alpha^2 +94 \right] \\
	& +  \frac{4}{3^4 35 (4n + 2\alpha +2)^6} \left[ 657 j_{\alpha,k}^6 + 27 j_{\alpha,k}^4 (73\alpha^2 - 181) \right. \\ 
	& \left. + j_{\alpha,k}^2 (2459\alpha^4 -10750\alpha^2 + 14051) +1493\alpha^6 - 9303\alpha^4 + 19887\alpha^2 - 12077 \right] \\
	& + \frac{1}{3^5 5^2 7 (4n + 2\alpha +2)^8} \left[ 215918\alpha^8 - 2292440\alpha^6 + 10190964\alpha^4 - 20174360\alpha^2 \right. \\
	&  + 12059918  + 53220 j_{\alpha,k}^8 + 240 j_{\alpha,k}^6 (887\alpha^2 - 2879)  + 3 j_{\alpha,k}^4 (125671\alpha^4 - 729422\alpha^2 + 1456807)  \\
	& \left. \left. + 6 j_{\alpha,k}^2 (63299\alpha^6 - 507801\alpha^4 + 1678761\alpha^2 -2201939)  \right] + \mathcal{O}(n^{-10}) \vphantom{\frac{2(\alpha^2 + j_{\alpha,k}^2 -1)}{3(4n + 2\alpha +2)^2}} \right).
\end{align}
Note that the $\mathcal{O}(n^{-10})$ term is inside the parenthesis, i.e., the asymptotic order is relative to the size of the leading order term.

\paragraph{Bulk.}
For the asymptotics in the bulk, one first has to determine the roots of the leading order term of the asymptotic expansion of the polynomials as expressed by \cref{Eleading_order}. For standard associated Laguerre this results in having to solve \cref{EtranscStdLag} to obtain $t_k$. A computational procedure is described in \cref{subsect_transcendental}.

The asymptotic expansion of the nodes in the bulk (region $\Ri$ in \cref{Fregions}) we arrive at is:
\begin{align}
	x_k & = ( 4 n +2 \alpha +2 ) t_k -\frac{ 1 }{ 12 ( 4 n + 2\alpha +2 ) } \left( \frac{ 5 }{ (1 -t_k)^2 } -\frac{ 4 }{ 1 -t_k } + 12 \alpha^2 -4 \right) \label{EnodeLagBulk} \\
	& + \frac{1-t_k}{720 t_k (4n + 2\alpha +2 )^3} \left(\frac{1600}{(1-t_k)^6} - \frac{3815}{(1-t_k)^5} + \frac{2814}{(1-t_k)^4} - \frac{576}{(1-t_k)^3} \right. \\
	& \left. -\frac{16}{(1-t_k)^2}  +16(15\alpha^4 - 30\alpha^2 + 7) \left[2-\frac{3}{1-t_k} \right] \right) \\
	& + \frac{-(1-t_k)^2}{2^6  3^4  35 t_k^2 (4n+2\alpha+2)^5} \left[\frac{-1727136}{(1-t_k)^5}+ \frac{16131880}{(1-t_k)^6} + \frac{-48469876}{(1-t_k)^7}  \right. \\
	& + 175 \left(\frac{379569}{(1-t_k)^8} + \frac{-246416}{(1-t_k)^9} + \frac{61700}{(1-t_k)^{10}} \right)  +  4608 \frac{3+2t_k}{1-t_k}\left(31 -147 \alpha^2 +105 \alpha^4 -21 \alpha^6 \right) \\
	& + \frac{384}{(1-t_k)^2}\left( -1346 +6405 \alpha^2 -4620 \alpha^4 +945 \alpha^6 \right) \\
	& \left.  +\frac{320}{(1-t_k)^3} \left( -43 +126 \alpha^2 -63 \alpha^4 \right)  +\frac{80}{(1-t_k)^4} \left( -221 -630 \alpha^2 +315 \alpha^4 \right)\right] \\
	& + \frac{(1-t_k)^3}{2^8  3^5 5^2 7 t_k^3 (4n+2\alpha+2)^7} \left[ \frac{43222750000}{(1-t_k)^{14} } - \frac{241928673000}{(1-t_k)^{13} } + \frac{566519158800}{(1-t_k)^{12} } \right. \\
	&  -\frac{714465642135}{(1-t_k)^{11} }+ \frac{518401904799}{ (1-t_k)^{10} } - \frac{212307298152}{ (1-t_k)^9 } + \frac{672}{ (1-t_k)^8 } \left( 12000\alpha^4 - 24000\alpha^2 +64957561 \right)  \\
	& - \frac{192}{(1-t_k)^7}\left(103425\alpha^4 -206850\alpha^2 + 15948182\right) + \frac{3360}{(1-t_k)^6}\left(4521\alpha^4 - 9042\alpha^2 - 7823\right) \\
	& - \frac{1792}{(1-t_k)^5}\left(3375\alpha^6 - 13905\alpha^4 + 17685\alpha^2 - 1598\right)+ \frac{16128}{(1-t_k)^4}\left(450\alpha^6 - 2155\alpha^4 + 2960\alpha^2 - 641\right) \\
	& -\frac{768}{(1-t_k)^3}\left(70875\alpha^8 - 631260\alpha^6 + 2163630\alpha^4 - 2716980\alpha^2 +555239\right) \\
	& + \frac{768}{(1-t_k)^2}\left(143325\alpha^8 - 1324260\alpha^6 + 4613070\alpha^4 -5826660\alpha^2 + 1193053\right) \\
	& - \frac{5806080}{1-t_k}\left(15\alpha^8 -140\alpha^6 + 490\alpha^4 - 620\alpha^2 + 127\right) + 24883200\alpha^8 -232243200\alpha^6 \\
	& \left. + 812851200\alpha^4 - 1028505600\alpha^2 +210677760 \vphantom{\frac{43222750000}{(1-t_k)^{14} }} \right] + \mathcal{O}(n^{-9}).
\end{align}
The first two terms in \cref{EnodeLagBulk} are exactly those found by Tricomi \cite[(56)]{tricBulk}, where the error bound was reported to be $\mathcal{O}(n^{-2})$, while \cite{gatt} reported a stricter $\mathcal{O}(n^{-3})$.

The corresponding expansion for the weights is
\begin{align}
	w_k & = x_k^\alpha \, \ee^{-x_k} 2\pi\sqrt{\frac{t_k}{1-t_k}} \left(1 - \frac{1}{6(4n +2\alpha +2)^2}\left[ 5 (1-t_k)^{-3} - 2(1-t_k)^{-2} \right]  \right. \label{EexplLagWeiBulk} \\ 
	&  + \frac{(1-t_k)^2}{720(4n +2\alpha +2)^4 t_k^2}\left[ \frac{8000}{(1-t_k)^8} - \frac{24860}{(1-t_k)^7} + \frac{27517}{(1-t_k)^6} - \frac{12408}{(1-t_k)^5} + \frac{1712}{(1-t_k)^4} \right. \\
	& \left. + \frac{32}{(1-t_k)^3} + \frac{16}{(1-t_k)^2} (15\alpha^4 - 30\alpha^2 + 7) \right] \\
	& - \frac{(1-t_k)^3}{90720(4n +2\alpha +2)^6 t_k^3} \left[\frac{43190000}{(1-t_k)^{12}} -\frac{204917300}{(1-t_k)^{11}} + \frac{393326325}{(1-t_k)^{10}} - \frac{386872990}{(1-t_k)^9} \right. \\
	& + \frac{201908326}{(1-t_k)^8} - \frac{50986344}{(1-t_k)^7}  +\frac{80}{ (1-t_k)^6 } \left( 315\alpha^4 - 630\alpha^2 + 53752 \right)  \\
	& - \frac{320}{ (1-t_k)^5 } \left(189\alpha^4 -378\alpha^2 - 89\right) + \frac{480}{ (1-t_k)^4 } \left( 63\alpha^4 - 126\alpha^2 + 43 \right) \\
	& -\frac{384}{(1-t_k)^3} \left( 315\alpha^6 - 1470\alpha^4 + 1995\alpha^2 - 416 \right) \\
	& \left. \left. + \frac{2304}{(1-t_k)^2} \left(21\alpha^6 -105\alpha^4 + 147\alpha^2 - 31\right) \right] \vphantom{\frac{1}{1} } + \mathcal{O}(n^{-8}) \right).
\end{align}

\paragraph{Soft edge.} The soft edge for the standard associated Gauss--Laguerre rule is around $x=4n$. The weights rapidly decay in this regime and even underflow for moderate $n$. For that reason we restrict the number of terms given here, but we note that more terms could be computed \cite{opsomer2018phd,codeQuadr}.

The expansion of the large nodes in region $\Rright$ is, where $a_k$ are the zeros of the Airy function,
\begin{align}
	x_k & = \left(4n+2\alpha+2\right) + 2^{2/3} a_{n-k+1} \left(4n+2\alpha+2\right)^{1/3} +\frac{2^{4/3}}{5} a_{n-k+1}^2 \left(4n+2\alpha+2\right)^{-1/3} \label{EnodeLagRight} \\ 
	& + \left(\frac{11}{35} -\alpha^2 -\frac{12}{175} a_{n-k+1}^3\right) \left(4n+2\alpha+2\right)^{-1} \\
	& + \left(\frac{16}{1575} a_{n-k+1} + \frac{92}{7875} a_{n-k+1}^4\right) 2^{2/3} \left(4n+2\alpha+2\right)^{-5/3} \\
	&  - \left(\frac{15152}{3031875} a_{n-k+1}^5 + \frac{1088}{121275} a_{n-k+1}^2 \right) 2^{1/3} \left(4n+2\alpha+2\right)^{-7/3} + \mathcal{O}(n^{-3}).
\end{align}
This expression was adapted from \cite[(4.9)]{gatt}, while we derived the following leading order term of the weight: 
\begin{equation}
 w_k = 4^{1/3} x_k^{\alpha+1/3} \exp(-x_k) \airy'(a_{n-k+1})^{-2} \left[ 1 + \mathcal{O}(n^{-2/3}) \right]. \label{EexplLagWeiRight}
\end{equation}
Here, $\airy'$ represents the derivative of the Airy function. Like above, the $\mathcal{O}$ term is in the parentheses, hence relative accuracy is obtained.

\subsection{Gauss--Jacobi}
\label{app:jacobi}

\paragraph{Hard edge at $x=-1$.}
Let $j_{\beta,k}$ be the $k$-th zero of the Bessel function $J_\beta$ of order $\beta$. Then we have the following terms, up to $\mathcal{O}(n^{-12})$:
\begin{align}
	x_k^{\qj\rl} & \sim -1 +\frac{2j_{\beta,k}^2}{ (2n+\alpha+\beta+1)^2} + \frac{-2j_{\beta,k}^2}{3 (2n+\alpha+\beta+1)^4} \left\{ \vphantom{\frac{-j_{\beta,k}^2}{6 (2n+\alpha+\beta+1)^5}} j_{\beta,k}^2 - 3 \alpha ^2 - \beta^2 + 1 \right\} \label{EexkStdJacBes} \\
	& + \frac{2 j_{\beta,k}^2}{45 (2n+\alpha+\beta+1)^6} \left\{ \vphantom{\frac{-j_{\beta,k}^2}{6 (2n+\alpha+\beta+1)^5}} 2 j_{\beta,k}^4 - 3j_{\beta,k}^2(5\alpha^2 +3\beta^2 - 2) +45\alpha^4 + 7\beta^4 + 20(3\alpha^2 - 1)\beta^2 - 60\alpha^2 + 13 \vphantom{\frac{-j_{\beta,k}^2}{6 (2n+\alpha+\beta+1)^5}} \right\} \\
	& + \frac{ - 2 j_{\beta,k}^2}{2835 (2n+\alpha+\beta+1)^8} \left\{ \vphantom{\frac{-j_{\beta,k}^2}{6 (2n+\alpha+\beta+1)^5}} 9j_{\beta,k}^6 - 18(7\alpha^2 + 5\beta^2 - 3)j_{\beta,k}^4 - 2835\alpha^6 - 247\beta^6 \right. \\
	&  + (328\beta^4 + [1512\alpha^2- 575]\beta^2 + 567\alpha^2 - 113)j_{\beta,k}^2   - 1407(3\alpha^2 - 1)\beta^4 + 8505\alpha^4 \\
	& \left. - 21(405\alpha^4 - 600\alpha^2 +133)\beta^2 - 8379\alpha^2 + 1633 \vphantom{\frac{-j_{\beta,k}^2}{6 (2n+\alpha+\beta+1)^5}} \right\} \\ 
	& + \frac{4 j_{\beta,k}^2}{42525 (2n+\alpha+\beta+1)^{10}} \left\{ \vphantom{\frac{-j_{\beta,k}^2}{6 (2n+\alpha+\beta+1)^5}}  6j_{\beta,k}^8 - 15(9\alpha^2 + 7\beta^2 - 4)j_{\beta,k}^6 + \left[6615\alpha^4 + 769\beta^4 + 2(1620\alpha^2 -589)\beta^2 \right. \right. \\
	& \left. - 12150\alpha^2 + 2668\right]j_{\beta,k}^4 + 3 \big[ 9450\alpha^6 - 999\beta^6 - 23(400\alpha^2 - 147)\beta^4 -40635\alpha^4 \\
	& - (2835\alpha^4 + 1850\alpha^2 - 294)\beta^2 + 50650\alpha^2 - 10236 \big] j_{\beta,k}^2 \\
	& + 42525\alpha^8 + 2327\beta^8 + 22340(3\alpha^2 - 1)\beta^6 - 226800\alpha^6 + 168(1530\alpha^4 - 2415\alpha^2 + 542)\beta^4 + 517860\alpha^4 \\
	& \left. + 20(11340\alpha^6 - 38745\alpha^4 + 42399\alpha^2 - 8488)\beta^2 - 509280\alpha^2 +98717 \vphantom{\frac{-j_{\beta,k}^2}{6 (2n+\alpha+\beta+1)^5}} \right\}  + \mathcal{O}(n^{-12}),
\end{align}

The corresponding weights are, up to $\mathcal{O}(n^{-8})$:
\begin{align}
	\frac{w_k^{\qj\rl}}{w(x_k^{\qj\rl})} & \sim \frac{8}{ J_{\beta-1}^2 (j_{\beta,k}) [2n+\alpha+\beta+1]^2} \bigg[ 1 + \frac{3 \alpha ^2 + \beta^2 - 1 -2j_{\beta,k}^2}{3 [2n+\alpha+\beta+1]^2} \label{EewkStdJacBes} \\
	& + \frac{ 45\alpha^4 + 7\beta^4 + 6 j_{\beta,k}^4 +20(3\alpha^2 - 1)\beta^2 - 6(5\alpha^2 + 3\beta^2 - 2) j_{\beta,k}^2 - 60\alpha^2 + 13  }{45 [2n+\alpha+\beta+1]^4}  \\
	& + \frac{1}{2835  [2n+\alpha+\beta+1]^6} \Big\{  2835\alpha^6 + 247\beta^6 - 36 j_{\beta,k}^6 + 1407(3\alpha^2 - 1)\beta^4  \\
	& +54(7\alpha^2 + 5\beta^2 - 3) j_{\beta,k}^4 - 8505\alpha^4 + 21(405\alpha^4 - 600\alpha^2 +133)\beta^2 \\
	& - 2\left[328\beta^4 + (1512\alpha^2 - 575)\beta^2 + 567\alpha^2 - 113\right]j_{\beta,k}^2 + 8379\alpha^2 - 1633 \Big\} + \mathcal{O}(n^{-8}) \bigg].
\end{align}

\paragraph{The bulk.}
The nodes in the interior of the interval expand asymptotically up to $\mathcal{O}(n^{-10})$ as
\begin{align}
	x_k & \sim t_k + \frac{2\alpha^2 - 2\beta^2 + (2\alpha^2 + 2\beta^2 - 1)t_k}{2[2n + \alpha + \beta + 1]^2}  + \frac{1}{24 (2n + \alpha + \beta + 1)^4 (1-t_k^2)} \label{EexkStdJacBulk} \\
	& \left(\vphantom{\frac{1^2}{x^2}} 32\alpha^4 - 32\beta^4 - \left[16\alpha^4 + 16\beta^4 + 4(12\alpha^2 -
5)\beta^2 - 20\alpha^2 + 5\right]t_k^3 - 24(\alpha^2 - \beta^2)t_k^2 - 40\alpha^2 + 40\beta^2 \right. \\
	& \left. + 3\left[16\alpha^4 + 16\beta^4 + 4(4\alpha^2 - 7)\beta^2 - 28\alpha^2 + 11\right]t_k \vphantom{\frac{1^2}{x^2}}\right) \\ 
	& + \frac{1}{240 (2n + \alpha + \beta + 1)^6 (1-t_k^2)^2} \left( \vphantom{\frac{1^2}{x^2}} 576\alpha^6 - 576\beta^6 \right. \\
	& + \left[96\alpha^6 + 96\beta^6 + 80(8\alpha^2 - 3)\beta^4 - 240\alpha^4 + 2(320\alpha^4 - 440\alpha^2 + 101)\beta^2 + 202\alpha^2 - 39\right]t_k^5  \\
	& -320(\alpha^2 - 6)\beta^4 + 240(\alpha^2 - \beta^2)t_k^4 - 1920\alpha^4 + 16(20\alpha^4 - 127)\beta^2 \\
	& -10\left[32(5\alpha^2 + 3)\beta^4 + 96\alpha^4 + 2(80\alpha^4 - 152\alpha^2 - 97)\beta^2- 194\alpha^2 + 99\right]t_k^3 \\
	& + 160\left[6\alpha^6 - 6\beta^6 +2(\alpha^2 + 15)\beta^4 - 30\alpha^4 - (2\alpha^4 + 41)\beta^2 + 41\alpha^2\right]t_k^2 + 2032\alpha^2 \\
	& \left. + 15\left[96\alpha^6 + 96\beta^6 + 16(4\alpha^2 - 23)\beta^4 - 368\alpha^4 + 2(32\alpha^4 - 72\alpha^2 + 223)\beta^2 + 446\alpha^2 - 173\right]t_k \vphantom{\frac{1^2}{x^2}} \right) \\ 
	& - \frac{1}{40320 (2n + \alpha + \beta + 1)^8 (1-t_k^2)^3} \left( \vphantom{\frac{1^2}{x^2}}  219648\alpha^8 - 219648\beta^8 - \big[9728\alpha^8 + 9728\beta^8 \right. \\
	& + 896(138\alpha^2 - 49)\beta^6 - 43904\alpha^6 + 224(1160\alpha^4 - 1720\alpha^2 + 389)\beta^4 \\
	& + 87136\alpha^4 + 8(15456\alpha^6 -48160\alpha^4 + 49364\alpha^2 - 9785)\beta^2 - 78280\alpha^2 + 14921\big]t_k^7 \\
	& - 10752(14\alpha^2 - 127)\beta^6 -40320(\alpha^2 - \beta^2)t_k^6 - 1365504\alpha^6 \\
	& + 21\big[2048\alpha^8 + 2048\beta^8 + 128(146\alpha^2 -123)\beta^6  - 15744\alpha^6 + 32(1320\alpha^4 - 1760\alpha^2 + 2023)\beta^4 + 64736\alpha^4 \\
	& + 8(2336\alpha^6- 7040\alpha^4 + 3644\alpha^2 - 11275)\beta^2 - 90200\alpha^2 + 37111\big]t_k^5 + 75264(5\alpha^2 - 49)\beta^4  \\
	& + 4480\big[44\alpha^8 - 44\beta^8 + 8(3\alpha^2 + 55)\beta^6 - 440\alpha^6 - 24(7\alpha^2 + 72)\beta^4 +1728\alpha^4 \\
	& - (24\alpha^6 - 168\alpha^4 - 2405)\beta^2 - 2405\alpha^2\big]t_k^4 \\
	& + 3687936\alpha^4 + 105\big[6656\alpha^8 + 6656\beta^8 - 128(50\alpha^2 + 443)\beta^6 - 56704\alpha^6 - 32(296\alpha^4 -696\alpha^2 - 6027)\beta^4 \\
	& + 192864\alpha^4 - 8(800\alpha^6 - 2784\alpha^4 + 3580\alpha^2 + 30285)\beta^2 -242280\alpha^2 + 99933\big]t_k^3 \\
	& + 384(392\alpha^6 - 980\alpha^4 + 10527)\beta^2 + 2688 \big[ 424\alpha^8 - 424\beta^8+ 4(4\alpha^2 + 783)\beta^6 \\
	& - 3132\alpha^6 + 4(35\alpha^2 - 2407)\beta^4 + 9628\alpha^4 - (16\alpha^6 +140\alpha^4 - 11429)\beta^2 - 11429\alpha^2 \big] t_k^2 \\
	&  - 4042368\alpha^2 + 35\big[23552\alpha^8 + 23552\beta^8 +128(90\alpha^2 - 1231)\beta^6 - 157568\alpha^6 + 32(328\alpha^4 - 1376\alpha^2 + 14095)\beta^4  \\
& \left. + 451040\alpha^4 + 8(1440\alpha^6 - 5504\alpha^4 + 9964\alpha^2 - 65439)\beta^2 - 523512\alpha^2 +206379\big]t_k \vphantom{\frac{1^2}{x^2}} \right) + \mathcal{O}(n^{-10}).
\end{align}

Up to $\mathcal{O}(n^{-8})$ relative accuracy, relative to the value of the weight function $w(x_k)$ at the corresponding node, the weights expand as
\begin{align}
	\frac{w_k}{w(x_k)} & \sim \frac{\pi\sqrt{1-t_k^2}}{2n + \alpha + \beta + 1} \bigg[ 2 - \frac{1 - 2\alpha^2 - 2\beta^2}{ (2n + \alpha + \beta + 1)^2}  \label{EewkStdJacBulk} \\
	&  +\frac{1}{ 12 (t_k^2-1)^2 (2n + \alpha + \beta + 1)^4} \bigg( \big[16\alpha^4 + 16\beta^4 + 4(12\alpha^2 - 5)\beta^2 - 20\alpha^2 + 5\big]t_k^4  + 48\alpha^4 + 48\beta^4 \\
	& + 12(4\alpha^2 - 7)\beta^2 - 6\big[4(4\alpha^2 + 1)\beta^2 + 4\alpha^2 - 3\big]t_k^2 - 84\alpha^2 + 64\big[\alpha^4 - \beta^4 - 2\alpha^2 + 2\beta^2\big]t_k +33\bigg) \\
	& +\frac{1}{120 (t_k^2-1)^3 (2n + \alpha + \beta + 1)^6} \bigg(\big[96\alpha^6 + 96\beta^6 + 80(8\alpha^2 - 3)\beta^4 - 240\alpha^4 \\
	& + 2(320\alpha^4 - 440\alpha^2 +101)\beta^2  + 202\alpha^2 - 39\big] t_k^6 - 1440\alpha^6 - 1440\beta^6 - 240(4\alpha^2 - 23)\beta^4 \\
	&  - 5\big[96\alpha^6+ 96\beta^6  + 16(20\alpha^2 - 27)\beta^4 - 432\alpha^4 + 2(160\alpha^4 - 136\alpha^2 + 295)\beta^2 + 590\alpha^2 - 237\big] t_k^4 \\
	& + 5520\alpha^4 - 640\big[3\alpha^6 - 3\beta^6 + (\alpha^2 + 15)\beta^4 - 15\alpha^4 -(\alpha^4 + 22)\beta^2 + 22\alpha^2\big]t_k^3 \\
	& - 30(32\alpha^4 - 72\alpha^2 + 223)\beta^2 - 15\big[288\alpha^6 +288\beta^6  - 16(8\alpha^2 + 81)\beta^4 - 1296\alpha^4 \\
	& - 2(64\alpha^4 - 88\alpha^2 - 863)\beta^2 +1726\alpha^2 - 717\big]t_k^2 \\
	&  - 6690\alpha^2 - 128\big[ 33\alpha^6 - 33\beta^6 - 5(\alpha^2 - 27)\beta^4 - 135\alpha^4 +(5\alpha^4 - 166)\beta^2 + 166\alpha^2 \big] t_k + 2595 \bigg)+ \mathcal{O}(n^{-8}) \bigg].
\end{align}

\bibliographystyle{abbrv}
\bibliography{references}

\end{document}

%% file: jacobi_accuracy.tikz
\begin{tikzpicture}[]
\begin{axis}[height = \figureheight, ylabel = {}, xmin = {7.0}, xmax = {400.0}, ymax = {1.0e-5}, ymode = {log}, xlabel = {}, unbounded coords=jump,scaled x ticks = false,xticklabel style={rotate = 0},xmajorgrids = false,xtick = {100.0,200.0,300.0,400.0},xticklabels = {$100$,$200$,$300$,$400$},xtick align = inside,axis lines* = left,scaled y ticks = false,yticklabel style={rotate = 0},log basis y=10,ymajorgrids = false,ytick = {1.0e-15,3.162277660168379e-13,1.0e-10,3.162277660168379e-8,1.0e-5},yticklabels = {$10^{-15.0}$,$10^{-12.5}$,$10^{-10.0}$,$10^{-7.5}$,$10^{-5.0}$},ytick align = inside,axis lines* = left,    xshift = 0.0mm,
    yshift = 0.0mm,
    axis background/.style={fill={rgb,1:red,1.00000000;green,1.00000000;blue,1.00000000}}
, ymin = {1.0e-16}, width = \figurewidth]\addplot+ [color = {rgb,1:red,0.00000000;green,0.60560316;blue,0.97868012},
draw opacity=1.0,
line width=1,
solid,mark = none,
mark size = 2.0,
mark options = {
    color = {rgb,1:red,0.00000000;green,0.00000000;blue,0.00000000}, draw opacity = 1.0,
    fill = {rgb,1:red,0.00000000;green,0.60560316;blue,0.97868012}, fill opacity = 1.0,
    line width = 1,
    rotate = 0,
    solid
}]coordinates {
(7.0, 3.175008367328758e-8)
(8.0, 2.5352480004237066e-8)
(9.0, 2.065436854881142e-8)
(10.0, 1.7122282369719244e-8)
(11.0, 1.3918812769020406e-9)
(12.0, 1.192608012345886e-9)
(13.0, 1.0322950272367848e-9)
(14.0, 9.016570823305869e-10)
(15.0, 7.939420232361272e-10)
(16.0, 7.04172276044801e-10)
(17.0, 6.286279274902995e-10)
(18.0, 5.644906764246116e-10)
(19.0, 2.2823598566645842e-10)
(20.0, 2.0724577609598782e-10)
(21.0, 1.8899215525891577e-10)
(22.0, 1.73024261584942e-10)
(23.0, 1.5897994032343377e-10)
(24.0, 2.5095037159417188e-11)
(25.0, 2.3225421585948425e-11)
(26.0, 2.1554869000794952e-11)
(27.0, 2.0056734051365765e-11)
(28.0, 1.87078130764462e-11)
(29.0, 1.7489898418432404e-11)
(30.0, 1.6385892642745148e-11)
(31.0, 1.538280613999632e-11)
(32.0, 1.4468093390007652e-11)
(33.0, 1.363253954167476e-11)
(34.0, 1.2866596676985864e-11)
(35.0, 1.216338141318829e-11)
(36.0, 7.038480909216105e-12)
(37.0, 6.674660824046441e-12)
(38.0, 6.338263247585019e-12)
(39.0, 6.026512622270275e-12)
(40.0, 5.7370774797504964e-12)
(41.0, 5.467848396278896e-12)
(42.0, 5.217160037318536e-12)
(43.0, 1.1622924844800764e-12)
(44.0, 1.1114442699522442e-12)
(45.0, 1.06403774680075e-12)
(46.0, 1.0191847366058937e-12)
(47.0, 9.773293285775253e-13)
(48.0, 9.380274335057948e-13)
(49.0, 9.009459844833145e-13)
(50.0, 8.663070261150096e-13)
(51.0, 8.333334022836425e-13)
(52.0, 8.022471575941381e-13)
(53.0, 7.72937269744034e-13)
(54.0, 7.450706718259426e-13)
(55.0, 7.187583861423263e-13)
(56.0, 6.938893903907228e-13)
(57.0, 6.70019595361282e-13)
(58.0, 6.477041125663163e-13)
(59.0, 4.601874437071274e-13)
(60.0, 4.453104551771503e-13)
(61.0, 4.3098857815948577e-13)
(62.0, 4.1744385725905886e-13)
(63.0, 4.0456527017340704e-13)
(64.0, 3.921307722976053e-13)
(65.0, 3.8047343053904115e-13)
(66.0, 3.6914915568786455e-13)
(67.0, 1.0158540675320182e-13)
(68.0, 9.869882688917642e-14)
(69.0, 9.592326932761353e-14)
(70.0, 9.303668946358812e-14)
(71.0, 9.059419880941277e-14)
(72.0, 8.815170815523743e-14)
(73.0, 8.559819519859957e-14)
(74.0, 8.348877145181177e-14)
(75.0, 8.126832540256146e-14)
(76.0, 7.915890165577366e-14)
(77.0, 7.72715225139109e-14)
(78.0, 7.527312106958561e-14)
(79.0, 7.338574192772285e-14)
(80.0, 7.149836278586008e-14)
(81.0, 6.994405055138486e-14)
(82.0, 6.816769371198461e-14)
(83.0, 6.661338147750939e-14)
(84.0, 6.505906924303417e-14)
(85.0, 6.361577931102147e-14)
(86.0, 6.217248937900877e-14)
(87.0, 6.050715484207103e-14)
(88.0, 4.929390229335695e-14)
(89.0, 4.829470157119431e-14)
(90.0, 4.707345624410664e-14)
(91.0, 4.6074255521943996e-14)
(92.0, 4.5075054799781356e-14)
(93.0, 4.418687638008123e-14)
(94.0, 4.3298697960381105e-14)
(95.0, 4.2299497238218464e-14)
(96.0, 1.3322676295501878e-14)
(97.0, 1.3211653993039363e-14)
(98.0, 1.2878587085651816e-14)
(99.0, 1.2434497875801753e-14)
(100.0, 1.2434497875801753e-14)
(100.0, 1.2434497875801753e-14)
(120.0, 8.659739592076221e-15)
(140.0, 1.9984014443252818e-15)
(160.0, 1.5543122344752192e-15)
(180.0, 6.661338147750939e-16)
(200.0, 6.106226635438361e-16)
(220.0, 6.106226635438361e-16)
(240.0, 5.551115123125783e-16)
(260.0, 6.661338147750939e-16)
(280.0, 6.661338147750939e-16)
(300.0, 5.551115123125783e-16)
(320.0, 5.551115123125783e-16)
(340.0, 4.440892098500626e-16)
(360.0, 4.440892098500626e-16)
(380.0, 4.440892098500626e-16)
(400.0, 3.885780586188048e-16)
};
\addlegendentry{Abs error $x_k$}
\addplot+ [color = {rgb,1:red,0.88887350;green,0.43564919;blue,0.27812294},
draw opacity=1.0,
line width=1,
solid,mark = none,
mark size = 2.0,
mark options = {
    color = {rgb,1:red,0.00000000;green,0.00000000;blue,0.00000000}, draw opacity = 1.0,
    fill = {rgb,1:red,0.88887350;green,0.43564919;blue,0.27812294}, fill opacity = 1.0,
    line width = 1,
    rotate = 0,
    solid
}]coordinates {
(7.0, 4.669665900736103e-8)
(8.0, 3.3906354328090244e-8)
(9.0, 2.5937919500700847e-8)
(10.0, 2.0574314411878754e-8)
(11.0, 1.806498395707473e-9)
(12.0, 1.4822939305458412e-9)
(13.0, 1.2412464311484086e-9)
(14.0, 1.0563851057759549e-9)
(15.0, 9.110789098269059e-10)
(16.0, 7.945442009030582e-10)
(17.0, 6.995034417174915e-10)
(18.0, 6.208806610090798e-10)
(19.0, 2.56966868991167e-10)
(20.0, 2.3061724731003546e-10)
(21.0, 2.0820403569710757e-10)
(22.0, 1.8896637958962223e-10)
(23.0, 1.7232217607670308e-10)
(24.0, 2.7903554055610556e-11)
(25.0, 2.5607710651966474e-11)
(26.0, 2.3589178931383963e-11)
(27.0, 2.180464033583826e-11)
(28.0, 2.0218168902256047e-11)
(29.0, 1.8801910662657652e-11)
(30.0, 1.7531166081741608e-11)
(31.0, 1.638711505773789e-11)
(32.0, 1.535251646497818e-11)
(33.0, 1.4414519544696724e-11)
(34.0, 1.3560568696818001e-11)
(35.0, 1.2781431692355988e-11)
(36.0, 7.468636583293815e-12)
(37.0, 7.0600776732349345e-12)
(38.0, 6.6846304121388074e-12)
(39.0, 6.338664817744969e-12)
(40.0, 6.019147031241784e-12)
(41.0, 5.723379601182404e-12)
(42.0, 5.4492138920073825e-12)
(43.0, 1.2273446363638135e-12)
(44.0, 1.1707626066493655e-12)
(45.0, 1.118251874601047e-12)
(46.0, 1.0688139357094658e-12)
(47.0, 1.0228602417568955e-12)
(48.0, 9.798772180033242e-13)
(49.0, 9.394761051898529e-13)
(50.0, 9.018530370915938e-13)
(51.0, 8.661680041472555e-13)
(52.0, 8.326263146767426e-13)
(53.0, 8.010894553038527e-13)
(54.0, 7.71192134436393e-13)
(55.0, 7.430319652576434e-13)
(56.0, 7.164784629579811e-13)
(57.0, 6.910596905009949e-13)
(58.0, 6.673364822207434e-13)
(59.0, 4.764449577964159e-13)
(60.0, 4.605118607472425e-13)
(61.0, 4.4521337729123224e-13)
(62.0, 4.3077253392907305e-13)
(63.0, 4.1706862787548976e-13)
(64.0, 4.03867594293224e-13)
(65.0, 3.9150786858716487e-13)
(66.0, 3.7952806028488306e-13)
(67.0, 1.0504289188698422e-13)
(68.0, 1.0195800747733163e-13)
(69.0, 9.899786906139572e-14)
(70.0, 9.593257635883623e-14)
(71.0, 9.333374814595607e-14)
(72.0, 9.074256779898706e-14)
(73.0, 8.804438317978159e-14)
(74.0, 8.580957226120345e-14)
(75.0, 8.346660677561356e-14)
(76.0, 8.124328115157586e-14)
(77.0, 7.925291244428838e-14)
(78.0, 7.715338113928819e-14)
(79.0, 7.51720943099152e-14)
(80.0, 7.31949451173927e-14)
(81.0, 7.156248516999999e-14)
(82.0, 6.970629819245579e-14)
(83.0, 6.808044471654292e-14)
(84.0, 6.64575797276125e-14)
(85.0, 6.495090318477345e-14)
(86.0, 6.344681073184718e-14)
(87.0, 6.171868316156839e-14)
(88.0, 5.041868696867942e-14)
(89.0, 4.937174036434546e-14)
(90.0, 4.8099765398668065e-14)
(91.0, 4.70565521067202e-14)
(92.0, 4.601502077809649e-14)
(93.0, 4.5088380334868256e-14)
(94.0, 4.4163172374826775e-14)
(95.0, 4.312613965618683e-14)
(96.0, 1.3629276442490893e-14)
(97.0, 1.5144484391505697e-14)
(98.0, 2.957507359938251e-14)
(99.0, 1.2703354010847947e-14)
(100.0, 1.269793562925449e-14)
(100.0, 1.269793562925449e-14)
(120.0, 1.9969999358399634e-14)
(140.0, 6.506003713206242e-14)
(160.0, 7.885909621873518e-14)
(180.0, 7.584067945589947e-14)
(200.0, 3.870156868590609e-14)
(220.0, 9.54424614001164e-14)
(240.0, 7.444058681081146e-14)
(260.0, 7.390136851084291e-14)
(280.0, 1.1115931593012331e-13)
(300.0, 7.206636824286516e-14)
(320.0, 6.217770638154632e-14)
(340.0, 1.415761884118766e-13)
(360.0, 1.404030507798482e-13)
(380.0, 7.826939813226203e-14)
(400.0, 1.7218540550491262e-13)
};
\addlegendentry{Rel error $x_k$}
\addplot+ [color = {rgb,1:red,0.24222430;green,0.64327509;blue,0.30444865},
draw opacity=1.0,
line width=1,
solid,mark = none,
mark size = 2.0,
mark options = {
    color = {rgb,1:red,0.00000000;green,0.00000000;blue,0.00000000}, draw opacity = 1.0,
    fill = {rgb,1:red,0.24222430;green,0.64327509;blue,0.30444865}, fill opacity = 1.0,
    line width = 1,
    rotate = 0,
    solid
}]coordinates {
(7.0, 1.505125810929986e-7)
(8.0, 1.0569850920905122e-7)
(9.0, 7.71855205722316e-8)
(10.0, 5.815307733275654e-8)
(11.0, 3.785664409727829e-8)
(12.0, 3.49825398093806e-8)
(13.0, 3.246348995178394e-8)
(14.0, 3.0245980109366855e-8)
(15.0, 2.828439982227593e-8)
(16.0, 2.654051450834416e-8)
(17.0, 2.4982505064663485e-8)
(18.0, 2.3583936858972976e-8)
(19.0, 1.7583411393573556e-9)
(20.0, 1.6729495011524875e-9)
(21.0, 1.594685855588196e-9)
(22.0, 1.5227704930342867e-9)
(23.0, 1.4565199324412248e-9)
(24.0, 1.3953358191098886e-9)
(25.0, 1.3386939334836256e-9)
(26.0, 1.2861351705861779e-9)
(27.0, 1.2372559088369428e-9)
(28.0, 1.1917017928020357e-9)
(29.0, 1.1491595175439073e-9)
(30.0, 1.109352859574031e-9)
(31.0, 1.0720364318483888e-9)
(32.0, 1.036992797187608e-9)
(33.0, 1.0040278053402574e-9)
(34.0, 9.729682892700708e-10)
(35.0, 9.436591508205083e-10)
(36.0, 1.3348158689474587e-10)
(37.0, 1.2973491725354336e-10)
(38.0, 1.261787341277909e-10)
(39.0, 1.2279957606331493e-10)
(40.0, 1.1958499468445183e-10)
(41.0, 1.1652367959413823e-10)
(42.0, 1.1360543061833539e-10)
(43.0, 1.1082072759460715e-10)
(44.0, 1.0816114670575416e-10)
(45.0, 1.0561834740130394e-10)
(46.0, 1.031852936428379e-10)
(47.0, 1.0085526858105709e-10)
(48.0, 9.862213823375043e-11)
(49.0, 9.647987964100935e-11)
(50.0, 9.442342741028398e-11)
(51.0, 9.244778553796351e-11)
(52.0, 9.05485963986763e-11)
(53.0, 8.872144685589944e-11)
(54.0, 8.696254827356142e-11)
(55.0, 8.526800099328824e-11)
(56.0, 8.363471026839875e-11)
(57.0, 8.20594009409703e-11)
(58.0, 8.053926969786573e-11)
(59.0, 1.6768184063487013e-11)
(60.0, 1.6470061425799543e-11)
(61.0, 1.6181764261880005e-11)
(62.0, 1.5902876238094166e-11)
(63.0, 1.5632772853990673e-11)
(64.0, 1.5371259820540217e-11)
(65.0, 1.5117865292957333e-11)
(66.0, 1.4872131304244363e-11)
(67.0, 1.4633863565371996e-11)
(68.0, 1.4402756765008462e-11)
(69.0, 1.417836681394391e-11)
(70.0, 1.3960499423149031e-11)
(71.0, 1.374880764792863e-11)
(72.0, 1.3543166588192435e-11)
(73.0, 1.3343229299245252e-11)
(74.0, 1.314874598090654e-11)
(75.0, 1.2959605610873837e-11)
(76.0, 1.2775544511178794e-11)
(77.0, 1.2596285126065254e-11)
(78.0, 1.2421765005488083e-11)
(79.0, 1.225164414364599e-11)
(80.0, 1.2086089073992667e-11)
(81.0, 1.1924586318379227e-11)
(82.0, 1.1767142815699572e-11)
(83.0, 1.1613592032500009e-11)
(84.0, 1.1463677229706093e-11)
(85.0, 1.1317530246301999e-11)
(86.0, 1.1174797198698627e-11)
(87.0, 1.1035533598047209e-11)
(88.0, 2.9769520182298947e-12)
(89.0, 2.940730992051499e-12)
(90.0, 2.9054536554440347e-12)
(91.0, 2.870946536059904e-12)
(92.0, 2.8370916727027407e-12)
(93.0, 2.804166621128701e-12)
(94.0, 2.771859131112109e-12)
(95.0, 2.740273286061523e-12)
(96.0, 2.709395208189136e-12)
(97.0, 2.679037547359542e-12)
(98.0, 2.6493876537081462e-12)
(99.0, 2.6204316494471414e-12)
(100.0, 2.5919127955020826e-12)
(100.0, 2.5919127955020826e-12)
(120.0, 2.1253485082972645e-12)
(140.0, 5.88792903322144e-13)
(160.0, 5.087701093753338e-13)
(180.0, 1.6814327707947996e-13)
(200.0, 1.4990092500610785e-13)
(220.0, 5.6538107529036097e-14)
(240.0, 2.418221872746429e-13)
(260.0, 1.5451875889915812e-13)
(280.0, 7.501638199514105e-14)
(300.0, 5.843589501175472e-14)
(320.0, 4.260133912303843e-14)
(340.0, 2.5908095113713614e-14)
(360.0, 1.3331349912881763e-14)
(380.0, 6.945832797811136e-15)
(400.0, 2.46001136128271e-14)
};
\addlegendentry{Abs error $w_k$}
\addplot+ [color = {rgb,1:red,0.76444018;green,0.44411178;blue,0.82429754},
draw opacity=1.0,
line width=1,
solid,mark = none,
mark size = 2.0,
mark options = {
    color = {rgb,1:red,0.00000000;green,0.00000000;blue,0.00000000}, draw opacity = 1.0,
    fill = {rgb,1:red,0.76444018;green,0.44411178;blue,0.82429754}, fill opacity = 1.0,
    line width = 1,
    rotate = 0,
    solid
}]coordinates {
(7.0, 9.964499280789502e-7)
(8.0, 9.853243183658103e-7)
(9.0, 9.777818818548953e-7)
(10.0, 9.72394787879588e-7)
(11.0, 8.848730363759729e-8)
(12.0, 8.808528786997009e-8)
(13.0, 8.777443301791891e-8)
(14.0, 8.752835015689916e-8)
(15.0, 8.732989298862802e-8)
(16.0, 8.716736691306462e-8)
(17.0, 8.703251658043086e-8)
(18.0, 8.69193537576246e-8)
(19.0, 4.1312207987837665e-8)
(20.0, 4.125387967213027e-8)
(21.0, 4.120366073311982e-8)
(22.0, 4.1160106176113324e-8)
(23.0, 4.112208397844041e-8)
(24.0, 6.8056695639620945e-9)
(25.0, 6.799255376355398e-9)
(26.0, 6.793567117769074e-9)
(27.0, 6.788497126716464e-9)
(28.0, 6.7839606673902056e-9)
(29.0, 6.779883011022544e-9)
(30.0, 6.776205720777317e-9)
(31.0, 6.772876244727178e-9)
(32.0, 6.769853637838399e-9)
(33.0, 6.767100263294936e-9)
(34.0, 6.764585242296563e-9)
(35.0, 6.762281592363728e-9)
(36.0, 4.259470044497177e-9)
(37.0, 4.2578472394511105e-9)
(38.0, 4.256351063067385e-9)
(39.0, 4.254968339053098e-9)
(40.0, 4.253686780512696e-9)
(41.0, 4.2524982498826535e-9)
(42.0, 4.251392862721144e-9)
(43.0, 9.70519194450245e-10)
(44.0, 9.70247017220423e-10)
(45.0, 9.699908807249972e-10)
(46.0, 9.697515855097894e-10)
(47.0, 9.695278317261497e-10)
(48.0, 9.693190487814126e-10)
(49.0, 9.691211911016904e-10)
(50.0, 9.689354622159827e-10)
(51.0, 9.687599726373276e-10)
(52.0, 9.685958376503637e-10)
(53.0, 9.684402417174297e-10)
(54.0, 9.682938981790435e-10)
(55.0, 9.68153335970659e-10)
(56.0, 9.680213375736209e-10)
(57.0, 9.67896166155109e-10)
(58.0, 9.677791712675746e-10)
(59.0, 7.24239261593541e-10)
(60.0, 7.24140944499082e-10)
(61.0, 7.240438484478051e-10)
(62.0, 7.239552194150512e-10)
(63.0, 7.238695393832126e-10)
(64.0, 7.237861982184907e-10)
(65.0, 7.23708716355982e-10)
(66.0, 7.236343571540667e-10)
(67.0, 2.0202562150127522e-10)
(68.0, 2.020036083982527e-10)
(69.0, 2.019816102918444e-10)
(70.0, 2.019608624193379e-10)
(71.0, 2.019402140048225e-10)
(72.0, 2.019215437781473e-10)
(73.0, 2.0190328435160538e-10)
(74.0, 2.0188507467170747e-10)
(75.0, 2.0186851456607984e-10)
(76.0, 2.0185271325162872e-10)
(77.0, 2.0183634127279703e-10)
(78.0, 2.0182133371749142e-10)
(79.0, 2.0180501485964016e-10)
(80.0, 2.0179287712149958e-10)
(81.0, 2.0177907120555977e-10)
(82.0, 2.0176622334281215e-10)
(83.0, 2.017539658529761e-10)
(84.0, 2.0174020608759537e-10)
(85.0, 2.017295575072396e-10)
(86.0, 2.0171799667383015e-10)
(87.0, 2.017086732053797e-10)
(88.0, 1.695171915479824e-10)
(89.0, 1.6950860408466265e-10)
(90.0, 1.6949679312313983e-10)
(91.0, 1.6948502263164337e-10)
(92.0, 1.694757289184601e-10)
(93.0, 1.6946835965663055e-10)
(94.0, 1.694597644350742e-10)
(95.0, 1.69450149777078e-10)
(96.0, 5.4296710489071885e-11)
(97.0, 5.4292075990209726e-11)
(98.0, 5.428901719332639e-11)
(99.0, 5.428769141077215e-11)
(100.0, 5.42831589134586e-11)
(100.0, 5.42831589134586e-11)
(120.0, 5.424175375118129e-11)
(140.0, 1.7538686134367144e-11)
(160.0, 1.7530209218484838e-11)
(180.0, 6.653072788878443e-12)
(200.0, 6.649460457808369e-12)
(220.0, 2.860679607106705e-12)
(240.0, 1.7267499255264495e-11)
(260.0, 1.2052332619206423e-11)
(280.0, 6.34985626853392e-12)
(300.0, 5.337747462225828e-12)
(320.0, 4.178692411036644e-12)
(340.0, 2.717177490856489e-12)
(360.0, 1.4892326771276847e-12)
(380.0, 1.1575368447054394e-12)
(400.0, 3.0870848737424554e-12)
};
\addlegendentry{Rel error $w_k$}
\end{axis}

\end{tikzpicture}